\documentclass[10pt,a4paper]{amsart}
\usepackage{amsmath,amssymb,fullpage,hyperref,verbatim,xspace}
\usepackage[nobysame]{amsrefs}
\usepackage[all,cmtip]{xy} %\CompileMatrices

\newtheorem{thm}{Theorem}[section]
\newtheorem{lem}[thm]{Lemma}
\newtheorem{prop}[thm]{Proposition}
\newtheorem{cor}[thm]{Corollary}
\newtheorem{add}[thm]{Addendum}
\theoremstyle{definition}
\newtheorem{rem}[thm]{Remark}
\newtheorem{defn}[thm]{Definition}
\newtheorem{ex}[thm]{Example}

\DeclareMathOperator{\Dhocolim}{Dhocolim}
\DeclareMathOperator{\hocolim}{hocolim}
\DeclareMathOperator{\hocolimiz}{hocolim^{iz}\!}
\DeclareMathOperator{\id}{id}
\DeclareMathOperator{\iz}{iz}

\DeclareMathOperator{\Spt}{Spt}

\DeclareMathOperator{\perm}{Perm}
\DeclareMathOperator{\strict}{Strict}
\DeclareMathOperator{\permnz}{Perm_{nz}}
\DeclareMathOperator{\strictnz}{Strict_{nz}}
\DeclareMathOperator{\permiz}{Perm_{iz}}
\DeclareMathOperator{\strictiz}{Strict_{iz}}
\DeclareMathOperator{\cat}{Cat}

\newcommand{\A}{\mathcal A}
\newcommand{\bfk}{\mathbf{k}}
\newcommand{\bfm}{\mathbf{m}}
\newcommand{\bfn}{\mathbf{n}}
\newcommand{\bfl}{\mathbf{l}}

\newcommand{\C}{\mathcal C}
\newcommand{\D}{\mathcal D}
\newcommand{\E}{\mathcal E}

\newcommand{\ew}{\overset{\sim}{\longleftarrow}}
\newcommand{\F}{\mathcal F}
\renewcommand{\geq}{\geqslant}
\newcommand{\ie}{\emph{i.e.}}
\renewcommand{\leq}{\leqslant}
\newcommand{\lla}{\longleftarrow}
\newcommand{\lra}{\longrightarrow}
\newcommand{\M}{\mathcal M}
\newcommand{\N}{\mathcal N}
\renewcommand{\P}{\mathcal P}
\newcommand{\Q}{\mathcal Q}
\newcommand{\R}{\mathcal R}
\newcommand{\vphi}{\varphi}
\newcommand{\vrho}{\varrho}
\newcommand{\we}{\overset{\sim}{\longrightarrow}}
\newcommand{\Z}{\mathbb{Z}}
\newdir{ >}{{}*!/-7pt/@{>}}
\renewcommand{\:}{\colon}

\def\dl{d_{\ell}}
\def\dr{d_{r}}
\def\rightdistributivity{right distributivity\xspace}
\def\Rightdistributivity{Right distributivity\xspace}
\def\leftdistributivity{left distributivity\xspace}

\newcommand{\bc}{\text{rig category}\xspace}
\newcommand{\bcs}{\text{rig categories}\xspace}
\newcommand{\bp}{\text{bipermutative category}\xspace}
\newcommand{\bps}{\text{bipermutative categories}\xspace}
\newcommand{\sbc}{strictly bimonoidal category\xspace}
\newcommand{\sbcs}{\text{strictly bimonoidal categories}\xspace}
\begin{document}
\title[Ring completion of rig categories]{
Ring completion of rig categories}
\author{Nils A. Baas, Bj{\o}rn Ian Dundas, Birgit Richter and John Rognes}
\address{Department of Mathematical Sciences, NTNU,
  7491 Trondheim, Norway}
\email{baas@math.ntnu.no}
\address{Department of Mathematics, University of Bergen,
  5008 Bergen, Norway}
\email{dundas@math.uib.no}
\address{Department Mathematik der Universit{\"a}t Hamburg,
  20146 Hamburg, Germany}
\email{richter@math.uni-hamburg.de}
\address{Department of Mathematics, University of Oslo,
  0316 Oslo, Norway}
\email{rognes@math.uio.no}
\date{\today}
\thanks{
The first author would like to thank the Institute for Advanced Study,
Princeton, for its hospitality and support during his stay in the
spring of 2007.
Part of the work was done while the second author was on sabbatical at
Stanford University, whose hospitality and stimulating environment is
gratefully acknowledged.
The third author thanks %%%010709br
the SFB 676 for support and the topology group
in Sheffield for stimulating discussions.}
\hyphenation{ca-te-go-ries ca-te-go-ry topo-logy}

\begin{abstract}
We offer a solution to the long-standing problem of group completing
within the context of rig categories (also known as bimonoidal
categories).  Given a rig category~$\R$ we construct a natural additive
group completion $\bar\R$ that retains the multiplicative structure,
hence has become a ring category.  If we start with a commutative
rig category~$\R$ (also known as a symmetric bimonoidal category) the
additive group completion $\bar\R$ will be a commutative ring category.
In an accompanying paper \cite{BDRR2} we show how to use this construction
to prove the conjecture from \cite{BDR} that the algebraic $K$-theory of
the connective topological $K$-theory ring spectrum~$ku$ is equivalent
to the algebraic $K$-theory of the rig category $\mathcal{V}$ of complex
vector spaces.
\end{abstract}

\maketitle

\section{Introduction and main result}

Multiplicative structure in algebraic $K$-theory is a delicate matter.
In 1980 Thomason \cite{Th2} demonstrated that, after additive group
completion, the most obvious approaches to multiplicative pairings
cease to make sense.  For instance, let us write $(-\M)\M$ for the
Grayson--Quillen \cite{GQ} model for the algebraic $K$-theory of a
symmetric monoidal category $\M$, written additively.  An object
in $(-\M)\M$ is a pair $(a,b)$ of objects of $\M$, thought of as
representing the difference ``$a-b$''.  The na{\"\i}ve guess for how
to multiply elements is then dictated by the rule that $(a-b)(c-d) =
(ac+bd)-(ad+bc)$.  This, however, does not lead to a decent multiplicative
structure: the resulting product is in most situations not functorial.

Several ways around this problem have been developed, but they all
involve first passing to spectra or infinite loop spaces by one of the
equivalent group completion machines, for instance the functor $\Spt$ from
symmetric monoidal categories to spectra defined in \cite{Th3}*{Appendix}.
The original problem has remained unanswered: can one additively group
complete and simultaneously keep the multiplicative structure, within
the context of symmetric monoidal categories?

We answer this question affirmatively.  Our motivation came from
an outline of proof in \cite{BDR} of the conjecture that $2$-vector
bundles give rise to a geometric cohomology theory of the same sort as
elliptic cohomology, or more precisely, to the algebraic $K$-theory of
connective topological $K$-theory, which by work of Ausoni and the fourth
author \cite{A}, \cite{AR} is a spectrum of telescopic complexity~$2$.
The solution of the ring completion problem given here enters
as a step in our proof in \cite{BDRR2} of that conjecture.  For this
application the alternatives provided in spectra were insufficient.

Before stating our main result, let us fix some terminology.

\begin{defn} \label{defn:unstable_stable_eq}
Let $|\C|$ denote the classifying space of a small category $\C$,
that is, the geometric realization of its nerve $N\C$.
A functor $F \: \C \to \D$ will be called an {\em unstable equivalence}
if it induces a homotopy equivalence of classifying spaces $|F| \: |\C|
\to |\D|$, and will usually be denoted $\C \we \D$.  A lax symmetric
monoidal functor $F \: \M \to \N$ of symmetric monoidal categories,
with or without zeros, is a \emph{stable equivalence} if it induces a
stable equivalence of spectra $\Spt F \: \Spt \M \to \Spt \N$.  Any lax
symmetric monoidal functor whose underlying functor is an unstable
equivalence is a stable equivalence.
\end{defn}

Unstable equivalences are often called homotopy equivalences, or weak
equivalences.  We use ``un\-stable'' to emphasize the contrast with stable
equivalences.  These definitions readily extend to simplicial categories
and functors between them.

By a rig (resp.~commutative rig) we mean a ring (resp.~commutative ring)
in the algebraic sense, except that negative elements are not assumed to
exist.  By a \emph{rig category} (resp.~\emph{commutative rig category}),
also known as a bimonoidal category (resp.~symmetric bimonoidal
category), we mean a category $\R$ with two binary operations
$\oplus$ and $\otimes$, satisfying the axioms of a rig (resp.~commutative
rig) up to coherent natural isomorphisms.  By a \emph{(simplicial)
ring category} we mean a (simplicial) rig category $\bar\R$ such that
$\pi_0 |\bar\R|$ is a ring in the usual sense, with additive inverses.
By a \emph{bipermutative category} (resp.~a \emph{strictly bimonoidal
category}) we mean a commutative \bc (resp.~a \bc) where as many of
the coherence isomorphisms as one can reasonably demand are identities.
See Definitions~\ref{def:igradedbim} and~\ref{def:igradedbc} below for
precise lists of axioms.

\begin{thm} \label{thm:main}
Let $(\R, \oplus, 0_\R, \otimes, 1_\R)$ be a small simplicial rig
category.  There are natural morphisms
$$
\R \ew Z\R \lra \bar\R
$$
of simplicial rig categories, such that
\begin{enumerate}
\item
$\bar\R$ is a simplicial ring category,
\item
$\R \ew Z\R$ is an unstable equivalence, and
\item
$Z\R \lra \bar\R$ is a stable equivalence.
\item
If furthermore (a) $\R$ is a groupoid, and (b) for every object $X$
in $\R$ the translation functor $X \oplus (-)$ is faithful, then there
is a natural chain of unstable equivalences of $Z\R$-modules connecting
$\bar\R$ to the Grayson--Quillen model $(-\R)\R$ for the additive group
completion of $\R$.
\end{enumerate}
\end{thm}

\begin{add}
\label{thm:mainadd}
Let $\R$ be a small simplicial commutative rig category.  There are
natural morphisms
$$
\R \ew Z\R \lra \bar\R
$$
of simplicial commutative rig categories, such that all four statements
of the theorem above hold.
\end{add}

In particular, the induced maps $\Spt \R \gets \Spt Z\R \to \Spt \bar\R$
are stable equivalences of ring spectra, but the point is that $\bar\R$
is ring complete, before passing to spectra.  Here are some 
examples of rig categories that can be ring completed by this method.

\begin{itemize}
\item
If $R$ is a rig, then the discrete category $\R$ with the elements of
$R$ as objects, and only identity morphisms, is a small rig category.
When $R$ is commutative, so is $\R$.  The spectrum $\Spt \R$ is the
Eilenberg--Mac\,Lane spectrum of the algebraic ring completion of $R$.
\item
There is a small commutative \bc $\E$ of finite sets, with objects
the finite sets $\bfn = \{1,\ldots,n\}$ for $n\geq0$.  In particular
$\mathbf0$ is the empty set.  There are no other morphisms in $\E$ than
the automorphisms, and the automorphism group of $\bfn$ is the symmetric
group $\Sigma_n$.  Disjoint union and cartesian product of sets induce
the operations $\oplus$ and $\otimes$, and $\Spt \E$ is equivalent to
the sphere spectrum.
\item
For each commutative ring $A$ there is a small commutative \bc $\F(A)$ of
finitely generated free $A$-modules.  The objects of $\F(A)$ are the free
$A$-modules $A^n = \bigoplus_{i=1}^n A$ for $n\geq0$.  There are no other
morphisms in $\F(A)$ than the automorphisms, and the automorphism group
of $A^n$ is the general linear group $GL_n(A)$.  Direct sum and tensor
product of $A$-modules induce the operations $\oplus$ and $\otimes$, and
$\Spt \F(A)$ is the (free) algebraic $K$-theory spectrum of the ring $A$.
\item
Let $\mathcal{V}$ be the topological commutative \bc of complex
(Hermitian) vector spaces.  It has one object $\mathbb{C}^n$ for each
$n\geq0$, with automorphism space equal to the unitary group $U(n)$.
There are no other morphisms.  The spectrum $\Spt \mathcal{V}$ is a
model for the connective topological $K$-theory spectrum $ku$.  The case
relevant to \cite{BDR} and \cite{BDRR2} is the $2$-category of $2$-vector
spaces of Kapranov and Voevodsky \cite{KV}, viewed as finitely generated
free $\mathcal{V}$-modules.  We can functorially convert $\mathcal{V}$
to a simplicial commutative rig category by replacing each morphism
space with its singular simplicial set.
\end{itemize}

\subsection{Outline of proof}
The problem should be approached with some trepidation, since the
reasons for the failure of the obvious attempts at a solution to this
long-standing problem in algebraic $K$-theory are fairly well hidden.
The standard approaches to additive group completion yield models that
are symmetric monoidal categories with respect to an additive structure,
but which have no meaningful multiplicative structure \cite{Th2}.
The failure comes about essentially because commutativity for addition
only holds up to isomorphism.  We therefore need to make a model that
provides enough room to circumvent this difficulty.

Our solution comes in the form of a graded construction, $G\R$, related
to iterations of the Grayson--Quillen model.  It is a $J$-shaped
diagram of symmetric monoidal categories, where the indexing category
$J=I\smallint\Q$ is a certain permutative category over the category $I$
of finite sets $\bfn = \{1, \dots, n\}$ and injective functions.
Its definition can be motivated in a few steps.  First, we use Thomason's
homotopy colimit \cite{Th3} of the diagram
$$
0 \lla \R \overset{\Delta}\lra \R \times \R
$$
in symmetric monoidal categories as a model for the additive group
completion of $\R$.  An object $(a,b)$ in the right hand category
$\R \times \R$ represents the difference $a-b$, while an object $a$
in the middle category~$\R$ represents the relation $a-a=0$, since $a$
maps to $(a,a)$ on the right hand side, and to zero in the
left hand category.

Group completion is a homotopy idempotent process, meaning that we
may repeat it any positive number of times and always obtain unstably
equivalent results.  For each $n\geq0$ we realize the $n$-fold iterated
group completion of $\R$ as the homotopy colimit of a $\Q\bfn$-shaped
diagram in symmetric monoidal categories, where $\Q\mathbf1$ is the
three-object category indexing the diagram displayed above, and in general
$\Q\bfn$ is isomorphic to the product of $n$ copies of $\Q\mathbf1$.
One distinguished entry in the $\Q\bfn$-shaped diagram is the product of
$2^n$ copies of $\R$.  Its objects are given by $2^n$ objects of~$\R$,
which we regard as being located at the corners of an $n$-dimensional
cube.  These represent an alternating sum of terms in~$\bar\R$, with
signs determined by the position in the $n$-cube.  The other entries
in the $\Q\bfn$-shaped diagram are diagonally embedded subcubes of the
$n$-cube, or the zero category, and encode cancellation laws in the
group completion.

As regards the multiplicative structure, there is a natural pairing
from the $n$-fold and the $m$-fold group completion
to the $(n+m)$-fold group completion, with all possible
$\otimes$-products of the entries in the two original cubes being spread
out over the bigger cube.  For instance, the product of the two $1$-cubes
$(a, b)$ and $(c, d)$ is a $2$-cube $\left(\smallmatrix a c & a d \\
b c & b d \endsmallmatrix\right)$, where for brevity we write $ac$
for $a \otimes c$, and so on.  Rather than trying to turn any single
$n$-fold group completion into a ring category, we instead
pass to the homotopy colimit over of all of them.  To allow
the homotopy colimit to retain the multiplicative structure, we proceed
as in \cite{Bo} and index the iterated group completions by
the permutative category $I$, instead of the (non-symmetric) monoidal
category of finite sets and inclusions that indexes sequential colimits.
For each morphism $\bfm \to \bfn$ in~$I$ there is a preferred functor from
$\Q\bfm$-shaped to $\Q\bfn$-shaped diagrams, involving extension by zero.
For instance, the unique morphism $\mathbf0 \to \mathbf1$ takes $a$
in~$\R$ (for $m=0$) to $(a, 0)$ in~$\R \times \R$ in the display above
(for $n=1$).  See section~\ref{sec:gq} for further examples and pictures
in low dimensions.

The resulting homotopy colimit, modulo a technical point about zero
objects, gives the desired ring category $\bar\R$.  As described, this
is the homotopy colimit of an $I$-shaped diagram, whose entry at $\bfn$
is the homotopy colimit of a $\Q\bfn$-shaped diagram, for each $n\geq0$.
Such a double homotopy colimit can be condensed into a single homotopy
colimit over a larger category, namely the Grothendieck construction
$J = I \smallint \Q$.  In the end we therefore prefer to present the
ring category $\bar\R$ as the one-step homotopy colimit of a $J$-shaped
diagram $G\R$.  The graded multiplication
$$
G\R(x)\times G\R(y)\to G\R(x+y)
$$
for $x, y$ in~$J$ is defined as above, by multiplying two cubes
together to get a bigger cube, and makes $G\R$ a \emph{$J$-graded
rig category}.  The difficulty one usually encounters does not appear,
essentially because we have spread the product terms out over the vertices
of the cubes, and not attempted to add together the ``positive'' and
``negative'' entries in some order or another.

From a homotopy theoretic point of view, the crucial information lies
in the fact that for each $n\geq0$, the homotopy colimit of the
spectra associated to the $\Q\bfn$-shaped part of the $G\R$-diagram is
stably equivalent to the spectrum associated with $\R$.  For instance,
the homotopy colimit of the diagram
$$
{*} = \Spt 0 \lla \Spt\R \overset{\Delta}\lra \Spt(\R\times\R)
$$
(for $n=1$) is the ``mapping cone of the diagonal'', hence is again a
model for the spectrum associated with~$\R$.  From a categorical point of
view, the possibility to interchange the factors in~$\R\times\R$ gives
that the passage to spectra is unnecessary, since this flip induces the
desired ``negative path components'', without having to stabilize.

We use Thomason's homotopy colimit in symmetric monoidal categories to
transform the $J$-graded rig category $G\R$ into the rig category $\bar\R$,
see Proposition~\ref{prop:GQisIgrded} and Lemma~\ref{lem:zeroless}.
The technical point alluded to above is that zero objects are troublesome
(few symmetric monoidal categories are ``well pointed''), and must be
handled with care.  This gives rise to the intermediate simplicial rig
category $Z\R$ that appears in Theorem~\ref{thm:main}.

\subsection{Plan}
The structure of the paper is as follows.  After replacing the starting
commutative rig (resp.~rig) category $\R$ by an equivalent bipermutative
(resp.~strictly bimonoidal) category, we discuss graded versions of
bipermutative and strictly bimonoidal categories and their morphisms
in section~\ref{sec:bip}.  In section~\ref{sec:gq} we introduce the
construction $G\R$ mentioned above, and show that it is a $J$-graded
bipermutative (resp.~strictly bimonoidal) category.

Thomason's homotopy colimit of symmetric monoidal categories is defined in
a non-unital (or zeroless) setting.  We extend this to the unital setting
by constructing a derived version of it in section~\ref{sec:hocolim}, and
in section~\ref{sec:hocolimgradedbps} we show that the homotopy colimit
of a graded bipermutative (resp.~graded strictly bimonoidal) category
is almost a bipermutative (resp.~strictly bimonoidal) category---it
only lacks a zero.  Section~\ref{sec:zeros} describes how the results
obtained so far combine to lead to an additive group completion within
the framework of (symmetric) bimonoidal categories.  This ring completion
construction is given in Theorem~\ref{thm:gcp}.

Most of this paper appeared earlier as part of a preprint \cite{BDRR} with
the title ``Two-vector bundles define a form of elliptic cohomology''.
Some readers thought that title was hiding the result on rig categories
explained in the current paper.  We therefore now offer the ring
completion result separately, and ask those readers interested in
our main application to also turn to \cite{BDRR2}.
One should note that there was a mathematical error in
the earlier preprint: the map~$T$ in the purported proof of Lemma~3.7(2) is not
well defined, and so the version of the iterated Grayson--Quillen model
used there might not have the right homotopy type.

A piece of notation: if $\C$ is any small category, then the expression
$X \in \C$ is short for ``$X$ is an object of $\C$'' and likewise for
morphisms and diagrams.  Displayed diagrams commute unless the contrary
is stated explicitly.

\section{Graded bipermutative categories} \label{sec:bip}

\subsection{Permutative categories}
A monoidal category (resp.~symmetric monoidal category) is a category~$\M$
with a binary operation $\oplus$ satisfying the axioms of a monoid
(resp.~commutative monoid), \ie, a group (resp.~abelian group) without
negatives, up to coherent natural isomorphisms.  A permutative category
is a symmetric monoidal category where the associativity and the left
and right unitality isomorphisms (but usually not the commutativity
isomorphism), are identities.  For the explicit definition of a
permutative category see for instance \cite{EM}*{3.1} or \cite{M1}*{\S 4}; compare also \cite{ML}*{XI.1}.  Since our permutative categories
are typically going to be the underlying additive symmetric monoidal
categories of categories with some further multiplicative structure,
we call the neutral element ``zero'', or simply $0$.

We consider two kinds of functors between permutative categories $(\M,
\oplus, 0_\M, \tau_\M)$ and $(\N, \oplus, 0_\N, \tau_\N)$, namely lax
and strict symmetric monoidal functors.  A \emph{lax symmetric monoidal
functor} is a functor $F$ in the sense of \cite{ML}*{XI.2}, \ie, there
are morphisms
$$
f(a,b) \: F(a) \oplus F(b) \rightarrow F(a \oplus b)
$$
for all objects $a,b \in \M$, which are natural in $a$ and $b$, there
is a morphism
$$
n\: 0_\N \rightarrow F(0_\M),
$$
and these structure maps fulfill the coherence conditions that are spelled
out in \cite{ML}*{XI.2}; in particular
$$
\xymatrix{
{F(a) \oplus F(b)} \ar[r]^-{f(a,b)} \ar[d]_{\tau_\N(F(a),F(b))} &
{F(a \oplus b)} \ar[d]^{F(\tau_\M(a,b))} \\
{F(b) \oplus F(a)} \ar[r]^-{f(b,a)} &
{F(b \oplus a)}}
$$
commutes for all $a,b \in \M$.
Let $\perm$ be the category of small permutative categories and
lax symmetric monoidal functors.

We might say that $f$ is a binatural transformation, \ie, a natural
transformation of functors $\M \times \M \to \N$.  Here ``bi-'' refers
to the two variables, and should not be confused with the ``bi-'' in
bipermutative, which refers to the two operations $\oplus$ and
$\otimes$.

A \emph{strict symmetric monoidal functor} has furthermore to satisfy that
the morphisms $f(a,b)$ and $n$ are identities, so that
$$
F(a \oplus b) = F(a) \oplus F(b) \quad \text{and} \quad F(0_\M) =
0_\N
$$
\cite{ML}*{XI.2}.  We denote the category of
small permutative categories and strict symmetric monoidal functors by
$\strict$.

A natural transformation $\nu \: F \Rightarrow G$ of lax symmetric
monoidal functors, with components $\nu_a \: F(a) \to G(a)$, is required
to be compatible with the structure morphisms, so that $\nu_{a \oplus
b} \circ f(a,b) = g(a,b) \circ (\nu_a \oplus \nu_b)$ and $\nu_{0_\M}
\circ n = n$.  Similarly for natural transformations of strict symmetric
monoidal functors.

Since any symmetric monoidal category is naturally equivalent
to a permutative category, we lose no generality by only
considering permutative categories.  We mostly consider the unital
situation, except for the places in sections~\ref{sec:hocolim}
and~\ref{sec:hocolimgradedbps} where we explicitly state to be in the
zeroless situation.

\subsection{Bipermutative categories}
Roughly speaking, a \bc $\R$ consists of a symmetric monoidal category
$(\R,\oplus,0_\R, \tau_\R)$ together with a functor $\R\times\R\to\R$
called ``multiplication'' and denoted by $\otimes$ or
juxtaposition.  Note that the multiplication is not a map of monoidal
categories.  The multiplication has a unit $1_\R \in\R$, multiplying by
$0_\R$ is the zero map, multiplying by $1_\R$ is the identity map, and
the multiplication is (left and right) distributive over $\oplus$ up to
appropriately coherent natural isomorphisms.  If we pose the additional
requirement that our \bcs are commutative, then this coincides with what
is often called a symmetric bimonoidal category.  Laplaza spelled out
the coherence conditions in \cite{L}*{pp.~31--35}.

According to \cite{M2}*{VI, Proposition~3.5} any commutative \bc is
naturally equivalent in the appropriate sense to a \bp, and a similar
rigidification result holds for rig categories.  Our main theorem
(resp.~its addendum) is therefore equivalent to the corresponding
statement where we assume that~$\R$ is a \sbc (resp.~a \bp).  We will
focus on the bipermutative case in the course of this paper, and indicate
what has to be adjusted in the strictly bimonoidal case.

The reader can recover the axioms for a \bp from Definition
\ref{def:igradedbim} below as the special case of a ``$0$-graded \bp'',
where $0$ is the one-morphism category.  Otherwise one may for instance
consult \cite{EM}*{3.6}.  One word of warning: Elmendorf and Mandell's
\leftdistributivity law is precisely what we (and \cite{M2}*{VI, Definition~3.3}) call the \rightdistributivity law.  Note that we
demand strict \rightdistributivity, and that this implies both cases of
condition~3.3(b) in \cite{EM}, in view of condition~3.3(c).

If $\R$ is a small \bc such that $\pi_0 |\R|$ is a ring (has additive
inverses), then we call $\R$ a \emph{ring category}.  Elmendorf and
Mandell's ring categories are not ring categories in our sense, but
non-commutative \bcs.  In the course of this paper we have to resolve \bcs
simplicially.  If~$\R$ is a small simplicial \bc such that $\pi_0 |\R|$ is
a ring, then we call $\R$ a \emph{simplicial ring category} (even though
it is usually not a simplicial object in the category of ring categories).

If $\R$ is a \sbc, a \emph{left $\R$-module} is a permutative category
$\M$ together with a multiplication $\R\times \M \to \M$ that is strictly
associative and coherently distributive, as spelled out in
\cite{EM}*{9.1.1}.

\subsection{$J$-graded \bps and \sbcs} \label{sec:igraded}
The following definition of a $J$-graded \bp is designed to axiomatize the
key properties of the functor $G\R$ described in section~\ref{sec:gq},
and simultaneously to generalize the definition of a \bp (as the case
$J = 0$).  More generally, we could have introduced $J$-graded \bcs
(resp.~$J$-graded commutative \bcs), generalizing \bcs (resp.~commutative
\bcs), but this would have led to an even more cumbersome definition.
We will therefore always assume that the input $\R$ to our machinery has
been transformed to an equivalent bipermutative or \sbc before we start.

\begin{defn} \label{def:igradedbim}
Let $(J, +, 0, \chi)$ be a small permutative category.
A \emph{$J$-graded \bp} is a functor
$$
\C \: J \lra \strict
$$
from $J$ to the category $\strict$ of small permutative categories and
strict symmetric monoidal functors, together with data $(\otimes, 1,
\gamma_\otimes)$ as specified below, and subject to the following
conditions.  The permutative structure of $\C(x)$ will be denoted
$(\C(x),\oplus, 0_x, \gamma_\oplus)$.

\begin{enumerate}
\item \label{igradedtensor}
There are composition functors
$$
\otimes \: \C(x)\times \C(y)\to \C(x+y)
$$
for all $x, y \in J$, that are natural in $x$ and~$y$.
More explicitly, for each pair of objects $a \in \C(x)$, $b \in \C(y)$
there is an object $a \otimes b$ in $\C(x+y)$, and for each pair of
morphisms $f \: a \to a'$, $g \: b \to b'$ there is a morphism $f \otimes
g \: a \otimes b \to a' \otimes b'$, satisfying the usual associativity
and unitality requirements.
For each pair of morphisms $k\: x \to z$, $\ell\: y \to w$ in $J$
the diagram
$$
\xymatrix{
{\C(x)\times \C(y)} \ar[r]^-{\otimes} \ar[d]_{\C(k)\times \C(\ell)}& {\C(x+y)}
\ar[d]^{\C(k + \ell)}\\
{\C(z)\times \C(w)} \ar[r]^-{\otimes} & {\C(z + w)}
}
$$
commutes.

\item \label{igradedunit}
There is a unit object $1 \in \C(0)$, such that $1 \otimes (-) \: \C(y) \to
\C(y)$ and $(-) \otimes 1 \: \C(x) \to \C(x)$ are the identity functors for
all $x,y \in J$.
More precisely, the inclusion $\{1\} \times \C(y) \to \C(0) \times \C(y)$
composed with $\otimes\: \C(0) \times \C(y) \to \C(0+y) = \C(y)$ equals
the projection isomorphism $\{1\} \times \C(y) \cong \C(y)$, and likewise
for the functor from $\C(x) \times \{1\}$.

\item \label{igradedmulttwist}
For each pair of objects $a \in \C(x)$, $b \in \C(y)$ there is a
twist isomorphism
$$
\gamma_\otimes =
\gamma_\otimes^{a,b}\: a\otimes b \lra \C(\chi^{y,x})(b \otimes a) 
$$
in $\C(x+y)$, where $\chi^{y,x} \: y+x \to x+y$, such that
$$
\xymatrix{
{a\otimes b} \ar[d]_{f \otimes g} \ar[r]^-{\gamma_\otimes^{a,b}} &
{\C(\chi^{y,x})(b \otimes a)} \ar[d]^{\C(\chi^{y,x})(g \otimes f)}\\
{a' \otimes b'} \ar[r]^-{\gamma_\otimes^{a',b'}}&
{\C(\chi^{y,x})(b' \otimes a')}
}
$$
commutes for any $f$, $g$ as above, and
$$
\C(k + \ell)(\gamma_\otimes^{a,b}) =
\gamma_\otimes^{\C(k)(a),\C(\ell)(b)} \,,
$$
for any $k$, $\ell$ as above.
We require that $\C(\chi^{y,x})(\gamma_\otimes^{b,a}) \circ
\gamma_\otimes^{a,b}$ is equal to the identity on $a \otimes b$ for all
objects $a$ and~$b$:
$$
\xymatrix{
{a\otimes b} \ar[rr]^-{\id_{a\otimes b}} \ar[dr]_-{\gamma_\otimes^{a,b}}
  && \C(\chi^{y,x})\C(\chi^{x,y})(a \otimes b) \\
& {\C(\chi^{y,x})(b \otimes a)}
  \ar[ur]_-{\C(\chi^{y,x})(\gamma_\otimes^{b,a})}
}
$$
and that $\gamma_\otimes^{a,1}$ and $\gamma_\otimes^{1,a}$ are equal to
the identity on $a$ for all objects~$a$.

\item \label{igradedmult}
The composition $\otimes$ is strictly associative, and the diagram
$$
\xymatrix{
{a \otimes b \otimes c} \ar[r]^{\gamma_\otimes^{a\otimes b,c}}
\ar[d]_{\id \otimes \gamma^{b,c}}&
{\C(\chi^{z,x+y})(c \otimes a \otimes b)}
\ar[d]^{\C(\chi^{z,x+y})(\gamma^{c,a} \otimes \id)}\\
\C(\id + \chi^{z,y})(a \otimes c \otimes b) \ar@{=}[r] &
\C(\chi^{z,x+y})\C(\chi^{x,z} + \id)(a \otimes c \otimes b)
}
$$
commutes for all objects $a$, $b$ and~$c$ (compare \cite{ML}*{p.~254, (7a)}).

\item \label{igradedzeros}
Multiplication with the zero object $0_x$ annihilates everything, for
each $x \in J$.
More precisely, the inclusion $\{0_x\}\times \C(y)\to \C(x)\times \C(y)$
composed with $\otimes \: \C(x)\times \C(y) \to \C(x+y)$ is the constant
functor to $0_{x+y}$, and likewise for the composite functor from $\C(x)
\times \{0_y\}$.

\item \label{igradedrightdist}
\Rightdistributivity holds strictly, \ie,
$$
\xymatrix{(\C(x)\times \C(x)) \times \C(y) \ar[r]^-{\oplus \times \id}
\ar[d]_\Delta&
\C(x)\times \C(y)\ar[dd]^\otimes\\
(\C(x)\times \C(y))\times(\C(x)\times \C(y))\ar[d]_{\otimes\times\otimes}&\\
\C(x+y)\times \C(x+y)\ar[r]^-{\oplus} & \C(x+y)
}
$$
commutes, where $\oplus$ is the monoidal structure and $\Delta$ is the
diagonal on $\C(y)$ combined with the identity on $\C(x)\times \C(x)$,
followed by a twist.  We denote these instances of identities
by $\dr$, so $\dr = \id \: \oplus \circ (\otimes \times \otimes) \circ
\Delta \to \otimes \circ (\oplus \times \id)$.

\item \label{igradedleftdist}
The \leftdistributivity transformation, $\dl$, is given in terms of
$\dr$ and $\gamma_\otimes$ as
$$
\dl = \gamma_\otimes \circ \dr \circ (\gamma_\otimes \oplus
\gamma_\otimes) \,.
$$
(Here we suppress the twist $\C(\chi)$ from the notation.)
More explicitly, for all $x, y \in J$ and $a \in \C(x)$, $b,b' \in \C(y)$
the following diagram defines $\dl$:
$$
\xymatrix{
{a \otimes b \oplus a \otimes b'} \ar[d]_{\dl}
\ar[rr]^-{\gamma_\otimes^{a,b}
\oplus \gamma_\otimes^{a,b'}} & & {\C(\chi^{y,x})(b \otimes a) \oplus
\C(\chi^{y,x})(b' \otimes a)} \ar@{=}[d] \\
{a \otimes (b \oplus b')} \ar@{=}[d]& & {\C(\chi^{y,x})(b \otimes a \oplus
b' \otimes a)}
\ar[d]^{\C(\chi^{y,x})(\dr) = \id} \\
{\C(\chi^{y,x})\C(\chi^{x,y})(a \otimes (b \oplus b'))}& &
{\C(\chi^{y,x})((b \oplus b') \otimes a)}
\ar[ll]_-{\C(\chi^{y,x})(\gamma_\otimes^{b \oplus b', a})}
}
$$

\item \label{igradedpermutation}
The diagram
$$
\xymatrix{
{(a \otimes b) \oplus (a \otimes b')} \ar[r]^-{\dl}
\ar[d]_{\gamma_\oplus} &
{a \otimes (b \oplus b')} \ar[d]^{\id \otimes \gamma_\oplus}\\
{(a \otimes b') \oplus (a \otimes b)} \ar[r]^-{\dl} & {a \otimes
  (b' \oplus b)}
}
$$
commutes for all objects.  The analogous diagram for $\dr$ also commutes.
Due to the definition of $\dl$ in terms of $\gamma_\otimes$
and the identity $\dr$, it suffices to demand that
$\gamma_\oplus \circ (\gamma_\otimes \oplus
\gamma_\otimes)=(\gamma_\otimes \oplus \gamma_\otimes)\circ \gamma_\oplus$
and
$(\gamma_\oplus\otimes \id)\circ \gamma_\otimes =\gamma_\otimes\circ
(\id\otimes \gamma_\oplus)$.

\item \label{igradeddistass}
The distributivity transformations are associative, \ie, the diagram
$$
\xymatrix{
{(a \otimes b \otimes c) \oplus (a \otimes b \otimes c')}
\ar[d]_{\dl} \ar[dr]^{\dl} & \\
{a \otimes ((b \otimes c) \oplus (b \otimes c'))} \ar[r]_-{\id
\otimes \dl} & {a \otimes b \otimes (c \oplus c')}
}
$$
commutes for all objects.
\item \label{igradedpentagon}
The following pentagon diagram commutes
$$
\xymatrix@C=-1.0cm{
& {(a \otimes (b \oplus b')) \oplus (a' \otimes (b \oplus b'))}
\ar[ddr]^{\dr}& \\
{(a \otimes b) \oplus (a \otimes b') \oplus (a' \otimes b) \oplus
(a' \otimes b')} \ar[ur]^{\dl\oplus \dl} \ar[dd]_{\id \oplus \gamma_\oplus
\oplus \id} & & \\
& & {(a \oplus a') \otimes (b \oplus b')} \\
{(a \otimes b) \oplus (a' \otimes b) \oplus (a \otimes b') \oplus
(a' \otimes b')} \ar[rd]_{\dr \oplus \dr}& & \\
& {((a \oplus a') \otimes b) \oplus ((a \oplus a') \otimes b')}
\ar[uur]_{\dl}&
}
$$
for all objects $a, a' \in \C(x)$ and $b, b' \in \C(y)$.
\end{enumerate}
\end{defn}

\begin{rem} \label{rem:igradedbim}
In Definition~\ref{def:igradedbim}, condition~\eqref{igradedtensor}
says that we have a binatural transformation
$$
\otimes\: \C\times \C\Rightarrow \C\circ (+)
$$
of bifunctors $J\times J\to \cat$, where $\cat$ denotes the category
of small categories.  Condition~\eqref{igradedmulttwist}
demands that there is a modification \cite{ML}*{p.~278}
$$
\xymatrix{
\C\times \C\ar@{=>}[r]^-{c_{\cat}}\ar@{=>}[d]^{\otimes}\ar@{}[dr]|-{{\overset{\gamma_{\otimes}}{\Rrightarrow}}}&
(\C\times \C)\circ tw_J\ar@{=>}[d]^{\otimes}\\
\C\circ (+) &\C\circ (+) \circ tw_J\ar@{=>}[l]^-{\C(\chi)}
}
$$
where $c_{\cat}$ is the symmetric structure on $\cat$ (with respect to
product) and $tw_J$ is the interchange of factors on $J\times J$.
\end{rem}

In the following we will denote a $J$-graded \bp $\C\: J \rightarrow
\strict$ by $\C^\bullet$ if the category~$J$ is clear from the context.
For the one-morphism category $J=0$, a $J$-graded bipermutative category
is the same as a \bp.  Thus every $J$-graded \bp $\C^\bullet$ comes with
%an underlying
a \bp $\C(0)$, and $\C^\bullet$ can be viewed as a functor $J\to
\C(0)\text{-modules}$.

\begin{ex} \label{ex:bipsets}
We consider the small \bp of finite sets, whose objects are the finite
sets of the form $\bfn = \{1,\ldots,n\}$ for $n \geq 0$, and whose
morphisms $\bfm \to \bfn$ are all functions
$\{1, \dots, m\} \to \{1, \dots, n\}$.

Disjoint union of sets gives rise to a permutative structure
$$
\bfn \oplus \bfm := \bfn \sqcup \bfm 
$$
and we identify $ \bfn \sqcup \bfm $ with $\mathbf{n+m}$.
For functions $f \: \bfn \to \bfn'$ and $g \: \bfm \to \bfm'$
we define $f \oplus g$ as the map on the disjoint union $f \sqcup g$
which we will denote by $f + g$.
The additive twist $c_{\oplus}$ is given by the shuffle maps
$$
\chi(n,m)\: \mathbf{n+m} \lra \mathbf{m+n}
$$
with
$$
\chi(n,m)(i) =
\begin{cases}
m+i & \text{for $i \leq n$} \\
i-n & \text{for $i > n$}.
\end{cases}
$$
Multiplication of sets is defined via
$$
\bfn \otimes \bfm := \mathbf{nm} \,.
$$
If we identify the element $(i-1)\cdot m+j$ in $\mathbf{nm}$ with the
pair $(i,j)$ with $i \in \bfn$ and $j \in \bfm$, then the function $f
\otimes g$ is given by
$$
(i,j) \mapsto (f(i),g(j)),
$$
and the multiplicative twist
$$
c_\otimes \: \bfn \otimes \bfm \lra \bfm \otimes \bfn
$$
sends $(i,j)$ to $(j,i)$.
The empty set $\mathbf0$ is a strict zero for the addition and the
singleton set $\mathbf1$ is a strict unit for the multiplication.
\Rightdistributivity is the identity and the \leftdistributivity law is
given by the resulting permutation
$$
\dl = c_\otimes \circ \dr \circ (c_\otimes \oplus c_\otimes) \,.
$$
For later reference we denote this instance of $\dl$ by $\xi$.

Considering only the subcategory of bijections, instead of arbitrary
functions, results in the bipermutative category of finite sets $\E$
that we referred to in the introduction.  Later, we will make use of the
zeroless \bp of finite nonempty sets and surjective functions.
\end{ex}

\begin{defn} \label{def:igradedbc}
A \emph{$J$-graded \sbc} is a functor $\C \: J \to \strict$ to the
category of permutative categories and strict symmetric monoidal
functors, satisfying the conditions of Definition~\ref{def:igradedbim},
except that we do not require the existence of the natural isomorphism
$\gamma_\otimes$, and the \leftdistributivity\xspace isomorphism $\dl$ is
not given in terms of $\dr$.  Axiom~\eqref{igradedleftdist} of Definition
\ref{def:igradedbim} has to be replaced by the following condition:

\begin{enumerate}
\item[(\ref{igradedleftdist}')]
The diagram
$$
\xymatrix{
{a \otimes b \otimes c \oplus a \otimes b' \otimes c} \ar[r]^-{\dr}
\ar[d]_{\dl} & {(a \otimes b \oplus a \otimes b') \otimes c}
\ar[d]^{\dl \otimes \id}\\
{a \otimes (b \otimes c \oplus b' \otimes c)} \ar[r]^-{\id \otimes
\dr} & {a \otimes (b \oplus b') \otimes c}
}
$$
commutes for all objects.
\end{enumerate}
\end{defn}

In the $J$-graded bipermutative case condition (\ref{igradedleftdist}')
follows from the other axioms.

\begin{defn} \label{def:laxmap}
A \emph{lax morphism of $\bps$}, $F \: \C \to \D$, is a pair of
lax symmetric monoidal functors $(\C,\oplus,0_\C, c_\oplus) \to (\D,
\oplus,0_\D, c_\oplus)$ and $(\C,\otimes,1_\C, c_\otimes) \to (\D,\otimes,1_\D,
c_\otimes)$, with the same underlying functor $F \: \C \to \D$, that
respect the left and right distributivity laws.

In other words, we have a binatural transformation from
$\oplus \circ (F \times F)$ to $F \circ \oplus$:
$$
\eta_\oplus = \eta_\oplus(a,b) \: F(a) \oplus F(b) \rightarrow
F(a \oplus b)
$$
for $a,b \in \C$, as well as a binatural transformation from
$\otimes \circ (F \times F)$ to $F \circ \otimes$:
$$
\eta_\otimes = \eta_\otimes(a,b) \: F(a) \otimes F(b) \rightarrow
F(a \otimes b)
$$
for $a,b \in \C$, plus morphisms $0_\D \to F(0_\C)$ and $1_\D \to
F(1_\C)$.  We require that these commute with $c_\oplus$ and $c_\otimes$,
respectively, and that the following diagram (and the analogous one for
$\dl$) commutes
$$
\xymatrix{
{F(a) \otimes F(b) \oplus F(a') \otimes F(b)} \ar[d]_{\eta_\otimes
\oplus \eta_\otimes} \ar@{=}[r]^-{\dr = \id}&
{(F(a)\oplus F(a'))\otimes F(b)} \ar[r]^-{\eta_\oplus \otimes \id}
&{F(a \oplus a') \otimes F(b)} \ar[d]^{\eta_\otimes}\\
{F(a\otimes b) \oplus F(a' \otimes b)} \ar[r]_{\eta_\oplus} &
{F(a\otimes b \oplus a' \otimes b)} \ar@{=}[r]_{F(\dr) =
\id} &{F((a \oplus a')\otimes b)}
}
$$
for all objects $a,a',b \in \C$, \ie, we have
$$
\eta_\oplus \circ
(\eta_\otimes \oplus\eta_\otimes) = \eta_\otimes \circ(\eta_\oplus \otimes
\id)
$$
and
$$
F(\gamma_\otimes
\circ (\gamma_\otimes \oplus \gamma_\otimes)) \circ \eta_\oplus \circ
(\eta_\otimes \oplus \eta_\otimes) =
\eta_\otimes \circ (\id \otimes \eta_\oplus) \circ
\gamma_\otimes \circ (\gamma_\otimes \oplus \gamma_\otimes) \,.
$$

For a \emph{lax morphism of $\sbcs$} we demand that $F$ is lax monoidal
with respect to $\otimes$, lax symmetric monoidal with respect to $\oplus$,
and that
$$
F(\dl) \circ \eta_\oplus \circ (\eta_\otimes \oplus
\eta_\otimes) = \eta_\otimes \circ (\id \otimes \eta_\oplus)\circ \dl
\quad \text{and} \quad
F(\dr) \circ \eta_\oplus \circ (\eta_\otimes \oplus
\eta_\otimes) = \eta_\otimes \circ (\eta_\oplus \otimes \id) \circ \dr.
$$
\end{defn}

\begin{defn}\label{def:gradedlaxmap}
A \emph{lax morphism of $J$-graded $\bps$}, $F\: \C^\bullet \rightarrow
\D^\bullet$, consists of a natural transformation $F$ from $\C^\bullet$
to $\D^\bullet$ that is compatible with the bifunctors $\oplus$, $\otimes$
and the units.  Additively, we require a transformation
$\eta_\oplus$ from $\oplus \circ (F\times F)$ to $F \circ \oplus$:
$$
\xymatrix{
{\C(x) \times \C(x)} \ar[r]^-{\oplus} \ar[d]_{F_x\times F_x} & {\C(x)}
\ar[d]^{F_x} \\
{\D(x) \times \D(x)} \ar[r]^-{\oplus} \ar@{=>}[ur]|{\eta_\oplus^x}
& {\D(x)}
}
$$
that commutes with $\gamma_\oplus$, is binatural with respect to morphisms
in $\C(x) \times \C(x)$, and is natural with respect to $x$.
Multiplicatively, we require a transformation
$\eta_\otimes$ from $\otimes \circ (F \times F)$ to $F \circ \otimes$:
$$
\xymatrix{
{\C(x) \times \C(y)} \ar[r]^-{\otimes} \ar[d]_{F_x\times F_y} & {\C(x+y)}
\ar[d]^{F_{x+y}} \\
{\D(x) \times \D(y)} \ar[r]^-{\otimes} \ar@{=>}[ur]|{\eta_\otimes^{x,y}}
& {\D(x+y)}
}
$$
that commutes with $\gamma_\otimes$, is binatural with respect to
morphisms in $\C(x) \times \C(y)$, and is natural with respect to $x$
and~$y$.  The functor $F$ must respect the distributivity constraints
in that it fulfills
$$
\eta_\oplus \circ (\eta_\otimes \oplus
\eta_\otimes) = \eta_\otimes \circ (\eta_\oplus \otimes \id)
$$
and
$$
F(\dl) \circ \eta_\oplus \circ (\eta_\otimes \oplus
\eta_\otimes) = \eta_\otimes \circ (\id \otimes \eta_\oplus) \circ \dl\,. 
$$

For a \emph{lax morphism of $J$-graded $\sbcs$} there is no requirement
on ($F$ and) $\eta_\otimes$ concerning the multiplicative twist
$\gamma_\otimes$.
\end{defn}

\section{A cubical construction on (bi-)permutative categories} \label{sec:gq}

We remodel the Grayson--Quillen construction \cite{GQ} of the group
completion of a permutative category to suit our multiplicative needs.
The na{\"\i}ve product $(ac \oplus bd, ad \oplus bc)$ of two pairs
$(a, b)$ and $(c, d)$ in their model will be replaced by the quadruple
$\left(\smallmatrix ac & ad \\ bc & bd\endsmallmatrix\right)$, where no
order of adding terms is chosen.  This avoids the ``phoniness'' of the
multiplication \cite{Th2}, but requires that we keep track of $n$-cubical
diagrams of objects, of varying dimensions $n\geq0$.  We start by introducing
the indexing category $I \smallint \Q$ for all of these diagrams, and
then describe the $I \smallint \Q$-shaped diagram~$G\M$ associated
to a permutative category~$\M$.  If we start with a bipermutative
category~$\R$, the result will be an $I \smallint \Q$-graded bipermutative
category~$G\R$.

\subsection{An indexing category} \label{sec:index}
Let $I$ be the usual skeleton of the category of finite sets and injective
functions, \ie, its objects are the finite sets $\bfn = \{1,\ldots,n\}$
for $n \geq 0$, and its morphisms are the injective functions $\vphi
\: \bfm \to \bfn$.  We define the sum of two objects $\bfn$ and $\bfm$
to be $\bfn + \bfm$ and use the twist maps $\chi(n,m)$ defined in Example
\ref{ex:bipsets}.  Then $(I, +, \mathbf0, \chi)$ is a permutative category.

For each $n\geq0$, let $\Q\bfn$ be the category whose objects are subsets $T$
of
$$
\{\pm 1, \dots, \pm n\} = \{-n, \dots, -1, 1, \dots, n\}
$$
such that the absolute value function $T\to\Z$ is injective.  In other
words, we may have $i \in T$ or $-i \in T$, but not both, for each $1 \leq
i \leq n$.
Morphisms in $\Q\bfn$ are inclusions $S \subseteq T$ of subsets.
(The objects could equally well be described as pairs $(T, w)$ where $T
\subseteq \bfn$ and $w$ is a function $T \to \{\pm1\}$, and similarly
for the morphisms.)
Let $\P\bfn \subseteq \Q\bfn$ be the full subcategory generated by
the subsets of $\bfn = \{1, \dots, n\}$, \ie, the $T$ with only positive
elements.

For example, the category $\Q\mathbf2$ can be depicted as:
$$
\xymatrix{
\{-1,2\} & \{2\} \ar[l] \ar[r] & \{1,2\} \\
\{-1\} \ar[u] \ar[d] & \varnothing \ar[l] \ar[r] \ar[u] \ar[d] & \{1\}
	\ar[u] \ar[d] \\
\{-1,-2\} & \{-2\} \ar[l] \ar[r] & \{1,-2\} \\
}
$$
and $\P\mathbf2$ is given by the upper right hand square.  We shall use
$\P\bfn$ and $\Q\bfn$ to index $n$-dimensional cubical diagrams with $2^n$
and $3^n$ vertices, respectively.

For each morphism $\vphi \: \bfm \to \bfn$ in $I$ we define a functor
$\Q\vphi \: \Q\bfm \to \Q\bfn$ as follows.  First, let $C\vphi = \bfn
\setminus \vphi(\bfm)$ be the complement of the image of the injective
function $\vphi$.  Then extend $\vphi$ to an odd function $\{\pm1,
\dots, \pm m\} \to \{\pm1, \dots, \pm n\}$, which we also call $\vphi$,
and let
$$
(\Q\vphi)(S) = \vphi(S) \sqcup C\vphi
$$
for each object $S \in \Q\bfm$.
For example, if $\vphi \: \mathbf1 \to \mathbf2$ is given by
$\vphi(1) = 2$, then $C\vphi = \{1\}$ and $\Q\vphi$ is the functor
$$
\xymatrix{
\{-1\} \ar@{|->}[d] & \varnothing \ar@{|->}[d] \ar[l] \ar[r]
	& \{1\} \ar@{|->}[d] \\
\{1, -2\} & \{1\} \ar[l] \ar[r] & \{1,2\}
}
$$
embedding $\Q\mathbf1$ into the right hand column of $\Q\mathbf2$.
Similarly, the function $\vphi \: \mathbf1 \to \mathbf2$ with
$\vphi(1)=1$ embeds $\Q\mathbf1$ into the upper row of $\Q\mathbf2$.

If $S \subseteq T$ then
clearly $(\Q\vphi)(S) \subseteq (\Q\vphi)(T)$.
If $\psi\:\bfk\to\bfm$ is a second morphism in $I$, we see
that $\Q\vphi\circ\Q\psi=\Q(\vphi\psi)$, and so $\bfn \mapsto
\Q\bfn$ defines a functor $\Q\: I\to \cat$.  Restricting to sets with
only positive entries, we get a subfunctor $\P\subseteq\Q$
that may be easier to grasp: if $\vphi\:\bfm\to\bfn\in I$, then
$\P\vphi\:\P\bfm \to\P\bfn$ is the functor
sending $S\subseteq\bfm$ to $\vphi(S) \sqcup C\vphi$, where $C\vphi=\bfn
\setminus \vphi(\bfm)$ is the complement of the image of $\vphi$.

Our main indexing category will be the Grothendieck construction
$J = I\smallint\Q$.  This is the category with objects pairs $x =
(\bfm,S)$ with $\bfm\in I$ and $S\in\Q\bfm$, and with morphisms $x =
(\bfm,S) \to (\bfn,T) = y$ consisting of pairs $(\vphi, \iota)$ with
$\vphi\:\bfm\to\bfn$ a morphism in $I$ and $\iota \: (\Q\vphi)(S)
\subseteq T$ an inclusion.
To give a functor $\C$ from $I \smallint \Q$ to any category is equivalent
to giving a functor $\C_\bfn$ from $\Q\bfn$ for each $n\geq0$, together
with natural transformations $\C_\vphi \: \C_\bfm \Rightarrow \C_\bfn
\circ \Q\vphi$ for all $\vphi \: \bfm \to \bfn$ in $I$, which must
be compatible with identities and composition in $I$.

Consider the functor $+\:\Q\bfn\times\Q\bfm\to
\Q(\mathbf{n+m})$ defined as follows.  The injective functions
$in_1\:\bfn\to\mathbf{n+m}$ and $in_2\:\bfm\to\mathbf{n+m}$ are
given by $in_1(i)=i$ and $in_2(j)=n+j$.  Extending to odd functions
we define $T+S$ to be the disjoint union of images
$$
in_1(T) \sqcup in_2(S) \subseteq \{\pm1,\dots,\pm(n+m)\} \,.
$$
For example, if $T=\{-1,2\}\subseteq\{\pm1,\pm2,\pm3\}$
and $S=\{1,-2\}\subseteq \{\pm1,\pm2\}$, then
$T+S=\{-1,2,4,-5\}\subseteq\{\pm1,\dots,\pm5\}$.

These functors, for varying $n,m\geq0$, combine to an
\emph{addition functor} on $I \smallint \Q$.  For each pair
of objects $(\bfn,T), (\bfm,S)\in I\smallint\Q$ we define
$(\bfn,T)+(\bfm,S)=(\mathbf{n+m},T+S)$, and
likewise on morphisms.

\begin{lem}
Addition makes $I\smallint\Q$ and $I\smallint\P$ into permutative
categories.
\end{lem}

\begin{proof}
The zero object is $(\mathbf0, \mathbf0)$, and the isomorphism
$(\chi(n,m), \id) \: (\mathbf{n+m},T+S) \to (\mathbf{m+n},S+T)$ provides
the symmetric structure.
\end{proof}

\subsection{The cube construction} \label{sec:cube}
Let $\M$ be a permutative category (with zero).  Define a functor
$$
\M_\bfn\: \P\bfn\to\strict
$$
for each $n\geq0$, by sending a subset $T\subseteq\bfn$ to $\M_{\bfn}(T)
= \M^{\P T}$, the permutative category of functions from the {\bf set}
$\P T$ of subsets of $T$ to $\M$, \ie, the product of one copy of $\M$
for each subset of $T$.  If $\iota \: S\subseteq T$ we get a strict
symmetric monoidal functor $\M_{\bfn}(\iota) \: \M^{\P S}\to \M^{\P
T}$ by sending the object $a=(a_U \mid U\subseteq S) \in \M^{\P S}$
to $(a_{V\cap S} \mid V\subseteq T)$, and likewise with morphisms.
These are diagonal functors, since each $a_U$ gets repeated once for
each $V$ with $V \cap S = U$.

For $n = 0,1,2$ the diagrams $\M_\bfn$ have the following shapes:
$$
\M \ ,\qquad \M \to \M\times\M \qquad\text{and}\qquad
\xymatrix{\M\times\M\ar[r]&\M^{\times4}\\\M\ar[u]\ar[r]&\M\times\M\ar[u]}
$$
where the morphisms are the appropriate diagonals.  In particular,
$\M_{\bfn}(\bfn)$ is the product of $2^n$ copies of~$\M$,
viewed as spread out over the corners of an $n$-dimensional cube.

For $\vphi\:\bfm\to\bfn$ we define a natural transformation
$\M_\vphi\:\M_\bfm\Rightarrow \M_\bfn\circ\P\vphi
$:
for $S\in\P\bfm$ we let $\M_\vphi(S)$ be the composite
$$
\M_\bfm(S)=\M^{\P S}\cong \M^{\P(\vphi(S))}\to\M^{\P(\vphi(S)\sqcup C\vphi)}=
\M_\bfn((\P\vphi)(S))
$$
where the isomorphism is just the reindexation induced by $\vphi$,
and the functor $\M^{\P(\vphi(S))}\to\M^{\P(\vphi(S)\sqcup C\vphi)}$ is the
identity on factors indexed by subsets of $\vphi(S)$ and zero on the
factors that are not hit by $\vphi$.  Explicitly,
$$
\M_\vphi(S)(f)_V = 
\begin{cases}
  f_{\vphi^{-1}(V)} & \text{if $V\subseteq\vphi(S)$} \\
  0 & \text{otherwise}
\end{cases}
$$
for any morphism $f\: a\to b\in \M^{\P S}$ and $V\subseteq\vphi(S)\sqcup
C\vphi$.  These are extension by zero functors, not diagonals.
Each $f_U$ gets repeated exactly once, as $\M_\vphi(S)(f)_V$ for $V =
\vphi(U)$.

For instance, if $\vphi\:\mathbf1\to\mathbf2$ is given by $\vphi(1)=2$,
then $\M_{\mathbf1}(\varnothing) = \M\to\M\times\M = \M_{\mathbf2}(\{1\})$
and $\M_{\mathbf1}(\mathbf1) = \M\times\M\to\M^{\times4} \cong
\M_{\mathbf2}(\mathbf2)$ are given by appropriate inclusions onto
factors in products.  For both morphisms $\vphi \: \mathbf1\to\mathbf2$
the associated functors $\M\to\M\times\M$ are the inclusion onto the
$\varnothing$-factor, whereas the two functors $\M\times\M\to\M^{\times4}$
include onto either the $\varnothing$ and $\{1\}$ factors, or the
$\varnothing$ and $\{2\}$ factors, depending on $\vphi$.

We see that for all $S\subseteq T\subseteq\bfm$ and $\vphi \: \bfm \to
\bfn$, the diagram
$$
\xymatrix{
  \M^{\P S} \ar[r] \ar[d]_{\M_\vphi(S)} & \M^{\P T} \ar[d]^{\M_\vphi(T)} \\
  \M^{\P(\vphi(S)\sqcup C\vphi)} \ar[r] & \M^{\P(\vphi(T)\sqcup C\vphi)}
}
$$
commutes, sending $a\in \M^{\P S}$ both ways to $W\mapsto
a_{\vphi^{-1}(W)\cap S}$ if $W\subseteq \vphi(T)$ and $0$ otherwise.

If $\psi\:\bfk\to\bfm\in I$, then we have an equality $\M_{\vphi\psi}
= \M_\vphi \M_\psi$ of natural transformations $\M_{\bfk} \Rightarrow
\M_\bfn \circ \P(\vphi\psi)$ in:
$$
\xymatrix{
\P\bfk \ar[d]_{\P\psi} \ar@(l,l)[dd]_{P(\vphi\psi)}
	\ar[dr]^{\M_{\bfk}} \\
\P\bfm \ar[d]_{\P\vphi} \ar[r]^-{\M_{\bfm}} & \strict \\
\P\bfn \ar[ur]_{\M_{\bfn}} \\
}
$$
%% $$
%% \xymatrix{
%% \P\bfk \ar[r]^{\P\psi} \ar[d]_{\M_{\bfk}} &
%% \P\bfm \ar[r]^{\P\vphi} \ar[d]|{\M_{\bfm}} &
%% \P\bfn \ar[d]^{\M_{\bfn}} \\
%% \strict \ar@{=}[r] \ar@{=>}[ur]|{\M_{\psi}} &
%% \strict \ar@{=}[r] \ar@{=>}[ur]|{\M_{\vphi}} &
%% \strict
%% }
%% $$
(This diagram is not strictly commutative.  The two right hand
triangles only commute up to the natural transformations $\M_{\psi}$
and $\M_{\vphi}$, respectively.)
Both natural transformations are represented by the functors $\M^{\P S}
\to \M^{\P(\vphi\psi(S)\sqcup C(\vphi\psi))}$ sending $a$ to $V\mapsto
a_{(\vphi\psi)^{-1}V}$ if $V\subseteq \vphi\psi(S)$ and $0$ otherwise.
Thus $\M$ can be viewed as a left lax transformation from the functor
$\P \: I \to \cat$ to the constant functor at $\strict$.
(We recall the definition of a left lax transformation in
subsection~\ref{sec:zeroless} below.)

This left lax transformation $\M \: \P \Rightarrow \strict$ extends to a
left lax transformation $\M \: \Q \Rightarrow \strict$ by declaring that
$\M_\bfn(T)=0$ if $T$ contains negative elements.

The first three diagrams now look like:
$$
\M \ ,\qquad 0 \gets \M \to \M\times\M \qquad\text{and}\qquad
\xymatrix{
0 & \M\times\M \ar[l]\ar[r] &\M^{\times4} \\
0 \ar[u]\ar[d] & \M \ar[u]\ar[d]\ar[r]\ar[l] & \M\times\M \ar[u]\ar[d] \\
0 & 0 \ar[l]\ar[r] & 0}
$$
Another way of saying that we have a left lax transformation $\Q
\Rightarrow\strict$ is to say that we have a functor $I\smallint\M\:
I\smallint\Q\to I\smallint\strict\cong I\times \strict$ (over $I$).
Projecting to the second factor, $I\smallint\M$ gives rise to a functor
$$
G\M\: I\smallint\Q\to\strict.
$$

Explicitly, $G\M(\bfn,T)=\M_\bfn(T)$, which is $\M^{\P T}$ if $T$
contains no negatives and $0$ otherwise.  If $\vphi\: \bfm\to\bfn\in
I$ and $\iota \:(\Q\vphi)(S)\subseteq T \in \Q\bfn$,
then $G\M(\vphi,\iota)\:\M_{\bfm}(S)\to\M_\bfn(T)$ is the composite of
$G\M(\vphi,\id)= \M_\vphi S\:\M_{\bfm}(S)\to\M_\bfn((\Q\vphi)(S))$ and
$G\M(\id,\iota)=\M_\bfn(\iota)\:\M_\bfn((\Q\vphi)(S))\to \M_\bfn(T)$.

\subsection{Multiplicative structure} \label{sec:multstr}
Since the diagram $G\M\: I\smallint\Q\to\strict$ is so
simple, only consisting of diagonals and inclusions onto factors in
products, algebraic structure on $\M$ is easily transferred to $G\M$.

\begin{prop}\label{prop:GQisIgrded} %%%010709br added \R
If $\R$ is a \sbc, then $G\R$ is an $I\smallint\Q$-graded \sbc.
If $\R$ is a \bp, then $G\R$ is an $I\smallint\Q$-graded \bp.
\end{prop}

\begin{proof}
We must specify composition functors
$$
\otimes \: G\R(\bfn,T) \times G\R(\bfm,S) \to G\R(\mathbf{n+m},T+S)
$$
for all $(\bfn,T), (\bfm,S) \in I \smallint \Q$.  Let $a\in
G\R(\bfn,T)$ and $b\in G\R(\bfm,S)$.  If $S$ and $T$ only contain positive
elements, then $a \otimes b \in G\R(\mathbf{n+m},T+S)$ is defined by
$$
(a \otimes b)_{V+U} = a_V \otimes b_U \,,
$$
where the $\otimes$-product on the right is formed in $\R$.  As $V$
and $U$ range over all the subsets of $T$ and $S$, respectively, $V+U$
ranges over all the subsets of $T+S$.  If $T$ or $S$ contain negative
elements, we set $a \otimes b = 0$.  Likewise for morphisms in
$G\R(\bfn,T)$ and $G\R(\bfm,S)$.  These composition functors are clearly
natural in $(\bfn,T)$ and $(\bfm,S)$.

The unit object $1$ of $G\R(\mathbf0, \mathbf0) \cong \R$ corresponds
to the unit object $1_{\R}$ of $\R$.  In the bipermutative case, the
twist isomorphism
$$
\gamma_\otimes \: a \otimes b \lra
	G\R(\chi(m,n),\id) (b \otimes a)
$$
has components
$$
(a \otimes b)_{V+U} = a_V \otimes b_U
  \overset{\gamma_\otimes^{\R}}\lra b_U \otimes a_V = (b \otimes a)_{U+V}
$$
for all $V \subseteq T$ and $U \subseteq S$,
where $\gamma_\otimes^{\R}$ is the twist isomorphism in $\R$.

Since everything is defined pointwise, the multiplicative structure on
$\R$ forces all the axioms of an $I\smallint\Q$-graded \sbc (or
$I \smallint \Q$-graded \bp) to hold for $G\R$.
\end{proof}

\section{Hocolim-lemmata} \label{sec:hocolim}

We recall Thomason's homotopy colimit construction in the case of
a $J$-shaped diagram of zeroless permutative categories, and then
construct a derived version of this construction for permutative
categories with zero.

\subsection{The case without zeros} \label{sec:zeroless}
Let $\permnz$ be the category of permutative categories without zero
objects, and lax symmetric monoidal functors.  Let $\strictnz$ be the
subcategory with the same objects, but with strict symmetric monoidal
functors as morphisms.  There are forgetful functors $V \: \strictnz
\to \permnz$ and $U \: \permnz \to \cat$, with composite $W = UV \:
\strictnz \to \cat$.

For any small category $J$, let $\cat^J$ be the category of functors $J
\to \cat$ and \emph{left lax transformations}.  Recall that for functors
$\C, \D \: J \to \cat$, a left lax transformation $F \: \C \to \D$ assigns
to each object $x \in J$ a functor $F_x \: \C(x) \to \D(x)$, and to each
morphism $k \: x \to y$ in $J$ a natural transformation $\nu^k \: \D(k)
\circ F_x \Rightarrow F_y \circ \C(k)$ of functors $\C(x) \to \D(y)$:
$$
\xymatrix{
\labelmargin={1pt}
\C(x) \ar[r]^{\C(k)} \ar[d]_{F_x} & \C(y) \ar[d]^{F_y} \\
\D(x) \ar[r]_{\D(k)} \ar@{=>}[ur]^{\nu^k} & \D(y)
}
$$
These must be compatible with composition in $J$, so that $\nu^{\id} =
\id$ and $\nu^{\ell k} = \nu^\ell\C(k) \circ \D(\ell)\nu^k$ for $\ell \:
y \to z$ in $J$.  If each $\nu^k = \id$, we have a natural transformation
in the usual sense.

Similarly, let $\permnz^{\!J}$ be the category of functors $J \to \permnz$
and left lax transformations.  In this case, the categories $\C(x),
\C(y), \D(x), \D(y)$ etc.~are symmetric monoidal without zero, the
functors $\C(k), \D(k), F_x, F_y$ etc.~are lax symmetric monoidal, and
$\nu^k$ is a natural transformation of lax symmetric monoidal functors.
Finally, let $\strictnz^{\!J}$ be the category of functors $J \to \strictnz$
and left lax transformations.  In this case, all of the symmetric
monoidal functors are strict.

Let $\Delta \: \cat \to \cat^J$ be the constant $J$-shaped diagram
functor.  Given a functor $\C \: J \to \cat$, the Grothendieck
construction $J \smallint \C$ is a model for the homotopy colimit in
$\cat$ \cite{Th1}.  We recall that an object in $J \smallint \C$ is a
pair $(x, X)$ where $x \in J$ and $X \in \C(x)$ are objects, while a
morphism $(x, X) \to (y, Y)$ is a pair $(k, f)$ where $k \: x \to y \in
J$ and $f \: \C(k)(X) \to Y \in \C(y)$ are morphisms.  This construction
defines a functor $J \smallint (-) \: \cat^J \to \cat$, which is left
adjoint to $\Delta \: \cat \to \cat^J$.  Here it is, of course, important
that we are allowing left lax transformations as morphisms in $\cat^J$,
since otherwise the left adjoint would be the categorical colimit.

Thomason's homotopy colimit of permutative categories \cite{Th3}
is constructed to have a similar universal property with respect to
the composite $\Delta V \: \strictnz \to \permnz^{\!J}$, where $V$
is as above and $\Delta \: \permnz \to \permnz^{\!J}$ is the constant
$J$-shaped diagram functor.  We briefly recall the explicit description.

\begin{defn}
Let $\C \: J \to \permnz$ be a functor.  An object in $\hocolim_J \C$
is an expression
$$
n[(x_1,X_1), \dots, (x_n,X_n)]
$$
where $n\geq1$ is a natural number, each $x_i$ is an object of $J$, and
each $X_i$ is an object of $\C(x_i)$.  A morphism from $n[(x_1,X_1),
\dots, (x_n,X_n)]$ to $m[(y_1,Y_1), \dots, (y_m,Y_m)]$ consists of three
parts: a surjective function $\psi \: \bfn \to \bfm$, morphisms $\ell_i
\: x_i \to y_{\psi(i)}$ in $J$ for each $1\leq i\leq n$, and morphisms
$\vrho_j \: \bigoplus_{i \in \psi^{-1}(j)} \C(\ell_i)(X_i) \to Y_j$
in $\C(y_j)$ for each $1\leq j\leq m$.  We will write $(\psi, \ell_i,
\vrho_j)$ to signify this morphism.
\end{defn}

See \cite{Th3}*{3.22} for the definition of composition in the category
$\hocolim_J \C$.  This category is permutative, without a zero, if
one defines addition to be given by concatenation \cite{Th3}*{p.~1632}.
Each left lax transformation $F \: \C \to \D$ induces a strict symmetric
monoidal functor $\hocolim_J F \: \hocolim_J \C \to \hocolim_J \D$,
so this construction defines a functor $\hocolim_J \: \permnz^{\!J} \to
\strictnz$.

The universal property in \cite{Th3}*{3.21} says that $\hocolim_J \:
\permnz^{\!J} \to \strictnz$ is left adjoint to $\Delta V \: \strictnz
\to \permnz^{\!J}$.  Again, it is critical that we are allowing
left lax transformations as morphisms in $\permnz^{\!J}$.

Recall Definition~\ref{defn:unstable_stable_eq} of unstable equivalences
in $\cat$ and stable equivalences in $\permnz$ and $\strictnz$.
We use the corresponding pointwise notions in diagram categories like
$\cat^J$ and $\permnz^{\!J}$, so a left lax transformation $F \: \C \to \D$
between functors $\C, \D \: J \to \permnz$ is a stable (resp.~unstable)
equivalence if every one of its components $F_x \: \C(x) \to \D(x)$
is a stable (resp.~unstable) equivalence, for $x \in J$.

\begin{lem} \label{lem:hocolim_is_a_homotopy_functor}
Let $F \: \C \to \D$ be a stable (resp.~unstable) equivalence in
$\permnz^{\!J}$.  Then
$$
\hocolim_J F \: \hocolim_J \C \to \hocolim_J \D
$$
is a stable (resp.~unstable) equivalence in $\strictnz$.

If $\C \: J \to \permnz$ is a constant functor and $J$ is contractible,
then $\C(x) \we \hocolim_J \C$ is an unstable equivalence for each $x
\in J$.
\end{lem}

\begin{proof}
The stable case follows from \cite{Th3}*{4.1}, since homotopy
colimits of spectra preserve stable equivalences.
The unstable case follows by the same line of argument, see
\cite{Th3}*{p.~1635} for an overview.

First consider the strict case, when $F \: \C \we \D$ is a left lax
transformation and unstable equivalence of functors $J \to \strictnz$.
The doubly forgetful functor $W \: \strictnz \to \cat$ has a left
adjoint, the free functor $P \: \cat \to \strictnz$, with $P\C =
\coprod_{n\geq1} \widetilde{\Sigma_n} \times_{\Sigma_n} \C^{\times n}$,
where $\widetilde{\Sigma_n}$ is the translation category of the symmetric
group $\Sigma_n$.

The free--forgetful adjunction $(P,W)$ generates a simplicial
resolution $(PW)^{\bullet+1}\C = \{[q] \mapsto (PW)^{q+1}\C\}$ of $\C$
by free zeroless permutative categories.  See \cite{Th3}*{1.2} for details.
By Thomason's argument \cite{Th3}*{p.~1641--1644}, the augmentation
$(PW)^{\bullet+1}\C \to \C$ induces an unstable equivalence
$$
\hocolim_J V(PW)^{\bullet+1}\C \we \hocolim_J V\C \,,
$$
and similarly for $\D$ and $F$.  Hence it suffices to prove that
$\hocolim_J V(PW)^{q+1} F$ is an unstable equivalence, for each $q\geq0$.

Let $\C' = W(PW)^q\C$, $\D' = W(PW)^q\D$ be functors $J \to \cat$,
and $F' = W(PW)^qF$ the resulting left lax transformation $\C' \to
\D'$.  Then $F'$ is an unstable equivalence, by $q$ applications of
Lemma~\ref{lem:free_preserves_we} below.  We must prove that $VPF' \:
VP\C' \to VP\D'$ induces an unstable equivalence of homotopy colimits.
This follows from Lemma~\ref{lem:defret} below, the fact that
the Grothendieck construction $J \smallint F' \: J \smallint \C' \to J
\smallint \D'$ respects unstable equivalences, and one more application
of Lemma~\ref{lem:free_preserves_we}.

Finally consider the lax case, when $F \: \C \we \D$ is a left lax
transformation and unstable equivalence of functors $J \to \permnz$.
For each $x \in J$ let $\hat\C(x) = \hocolim_0 \C(x)$ be the homotopy
colimit of the functor $0 \to \permnz$ taking the unique object of $0$
to $\C(x)$.  By the universal property of $\hocolim_0$ this defines a
functor $\hat\C \: J \to \strictnz$, and a natural transformation $\C \to
V\hat\C$.  It is an unstable equivalence by~\cite{M4}*{4.3}.
Summation in the permutative categories $\C(x)$ induces a left lax
natural transformation $V\hat\C \to \C$, such that
the composite $\C \to V\hat\C \to \C$ equals the identity transformation.

By naturality of these constructions with respect to $F$, we see that
$F \: \C \to \D$ is a retract of $V\hat F \: V\hat\C \to V\hat\D$ as a
morphism in $\permnz^{\!J}$, where $\hat F \: \hat\C \we \hat\D$ is a left lax
transformation and unstable equivalence of functors $J \to \strictnz$.
By functoriality, $\hocolim_J F$ is a retract of $\hocolim_J V\hat F$,
which is an unstable equivalence by the first case of the proof applied
to $\hat F \: \hat\C \to \hat\D$.  It follows that $\hocolim_J F$ is
also an unstable equivalence.

The claim in the case of a constant diagram follows by the same argument.
\end{proof}

\begin{lem} \label{lem:free_preserves_we}
The free functor $P \: \cat \to \strictnz$ sends unstable equivalences
to unstable equivalences.
\end{lem}

\begin{proof}
This follows from the natural homeomorphism of classifying spaces $|P\C|
\cong \coprod_{n\geq1} E\Sigma_n \times_{\Sigma_n} |\C|^{\times n}$
and the fact that $|\widetilde{\Sigma_n}| = E\Sigma_n$ is a free
$\Sigma_n$-space.
\end{proof}

\begin{lem} \label{lem:defret}
Let $\C' \: J \to \cat$ be any functor.  There is a natural unstable
equivalence
$$
P(J \smallint \C') \we \hocolim_J VP\C' \,.
$$
\end{lem}

\begin{proof}
Thomason proved this in \cite{Th3}.  There the statement appears in the
second paragraph on page~1639, in rather different--looking notation,
and the proof starts with the last paragraph on page~1637.
\end{proof}

\begin{lem} \label{lem:boekstedt_lemma}
Let $I$ be the category of finite sets and injective functions,
and let $\bfm \in I$.  If $\C \: I \to \permnz$ is a functor that
sends each $\vphi \: \bfm \to \bfn \in I$ to a stable
(resp.~unstable) equivalence $\C(\vphi) \: \C(\bfm) \to \C(\bfn)$,
then the canonical functor $\C(\bfm) \to \hocolim_I \C$ is a
stable (resp.~unstable) equivalence.
\end{lem}

\begin{proof}
This is a weak version of B{\"o}kstedt's lemma \cite{Bo}*{9.1}, which
holds for homotopy colimits in $\cat$ since it holds for homotopy colimits
in simplicial sets.  By the argument above, using the resolution by free
permutative categories, it also holds in $\permnz$.
\end{proof}

\subsection{The case with zero} \label{sec:withzero}
We shall need a version of the homotopy colimit for diagrams of
permutative categories with zero.  Thomason comments that such a homotopy
colimit with zero is not a homotopy functor, unless the category is
``well based''.  Hence we must derive our functor to get a homotopy
invariant version.  We do this by means of another simplicial
resolution, this time generated by the free--forgetful adjunction
between permutative categories with and without zeros.

The functor $R \: \strict \to \strictnz$ that forgets the special role
of the zero object has a left adjoint $L \: \strictnz \to \strict$, given
by adding a disjoint zero object: $L\N = \N_+$ for $\N \in \strictnz$.
Likewise, the forgetful functor $R' \: \perm^J \to \permnz^{\!J}$ has a left
adjoint $L' \: \permnz^{\!J} \to \perm^J$, given by adding disjoint zeros
pointwise: $L'\C \: x \mapsto \C(x)_+$ for $\C \: J \to \permnz$
and all $x \in J$.

Let $\strictiz \subset \strict$ be the full subcategory generated by
objects of the form $L\N = \N_+$, \ie, the permutative categories with
an \emph{isolated zero} object.  Similarly, let $\permiz^{\!J} \subset
\perm^J$ be the full subcategory generated by objects of the form $L'\C
= \C_+$.

In the statement and proof of following lemma we omit the forgetful
functors $R$ and $R'$ from the notation, and write $\N_+$ and $\C_+$
for $L\N$ and $L'\C$, respectively.
Note that $\Delta V (\N_+) = \Delta V(\N)_+$, where $\Delta V \:
\strictnz \to \permnz^{\!J}$ is as in subsection~\ref{sec:zeroless}.

\begin{lem}
The functor $(\Delta V)_{\iz} \: \strictiz \to \permiz^{\!J}$, taking $\N_+$ to
the constant diagram $x \mapsto \N_+$ for $x \in J$, has a left adjoint
$\hocolimiz_J \: \permiz^{\!J} \to \strictiz$, satisfying
$$
\hocolimiz_J (\C_+) = (\hocolim_J \C)_+
$$
for all zeroless diagrams $\C \: J \to \permnz$.
\end{lem}

\begin{proof}
To define the functor $\hocolimiz_J$, we must specify a strict symmetric
monoidal functor
$$
\hocolimiz_J  F \: (\hocolim_J \C)_+ \lra (\hocolim_J \D)_+
$$
for each left lax transformation $F \: \C_+ \to \D_+$.
The morphism
$$
\eta_+ \: \D_+ \to (\Delta V(\hocolim_J \D ))_+
= \Delta V((\hocolim_J \D)_+) \,,
$$
where $\eta \: \id \to \Delta V \circ \hocolim_J$ is the adjunction
unit, induces a function
\begin{align*}
\perm^J(\C_+, \D_+) &\cong \permnz^{\!J}(\C, \D_+)
\to \permnz^{\!J}(\C, \Delta V((\hocolim_J \D)_+)) \\
&\cong \strictnz(\hocolim_J \C, (\hocolim_J D)_+)
\cong \strict((\hocolim_J \C)_+, (\hocolim_J D)_+) \,.
\end{align*}
We declare the image of $F \: \C_+ \to \D_+$ to be $\hocolimiz_J
F$.  A diagram chase shows that $\hocolimiz_J (GF) = \hocolimiz_J G
\circ \hocolimiz_J F$ for each $G \: \D_+ \to \E_+$, so $\hocolimiz_J$
is a functor.

The adjunction property follows from the chain of natural
bijections
\begin{align*}
\strict(\hocolimiz_J(\C_+), \N_+)
&= \strict((\hocolim_J \C)_+, \N_+)
	\cong \strictnz(\hocolim_J \C, \N_+) \\
&\cong \permnz^{\!J}(\C, \Delta V(\N_+))
	= \permnz^{\!J}(\C, (\Delta V)_{\iz}(\N_+)) \\
&\cong \perm^J(\C_+, (\Delta V)_{\iz}(\N_+)) \,.  \qedhere
\end{align*}
\end{proof}

\begin{defn}
For each permutative category with zero $\M \in \perm$ let $Z\M \in
\permiz^{\Delta^{op}} \subset \perm^{\Delta^{op}}$ be the simplicial
object in permutative categories with isolated zeroes given by
$$
[q] \mapsto Z_q\M = (LR)^{q+1}(\M) \,.
$$
The face and degeneracy maps are induced by the adjunction counit $LR
\to \id$ and unit $\id \to RL$, as usual.  The counit also induces
a natural augmentation map $\epsilon \: Z\M \to \M$ of simplicial
permutative categories with zero, where $\M$ is viewed as a constant
simplicial object.
\end{defn}

\begin{lem} \label{lem:epsunstable}
Let $\M$ be a permutative category.  The augmentation map $\epsilon \:
Z\M \we \M$ is an unstable equivalence.
\end{lem}

\begin{proof}
The map $R\epsilon \: RZ\M \to R\M$ of simplicial zeroless permutative
categories admits a simplicial homotopy inverse, induced by the adjunction
unit.  Hence the map of classifying spaces $|\epsilon| \: |Z\M| \to |\M|$
admits a homotopy inverse, since the classifying space only depends on
the underlying category.
\end{proof}

We extend $Z$ pointwise to define a simplicial resolution $\epsilon \:
Z\C \we \C$ for any $\C \: J \to \perm$, with $Z\C \: x \mapsto Z\C(x)$
for all $x\in J$.  This allows us to define a derived homotopy colimit
for permutative categories with zero.

\begin{defn}
The \emph{derived homotopy colimit}
$$
\Dhocolim_J \: \perm^J \to \strictiz^{\Delta^{op}} \subset
	\strict^{\Delta^{op}}
$$
is defined by
$$
\Dhocolim_J \C = \hocolimiz_J (Z\C)
= \{ [q] \mapsto \hocolimiz_J (LR)^{q+1}\C \} \,.
$$
\end{defn}

The construction deserves its name:

\begin{lem} \label{lem:dhocolim_is_derived}
Let $\C \to \D$ be a stable (resp.~unstable) equivalence
in $\perm^J$.  Then $Z_q \C \to Z_q \D$ is a
stable (resp.~unstable) equivalence for each $q\geq0$, so the
induced functor
$$
\Dhocolim_J \C \lra \Dhocolim_J \D
$$
is a stable (resp.~unstable) equivalence, too.
\end{lem}

\begin{proof}
The functor $LR$ adds a disjoint base point to the classifying space,
and the counit $LR \to \id$ induces a stable equivalence of spectra
\cite{Th3}*{2.1}.  Hence $LR$ preserves both stable and unstable
equivalences.  Iterating $(q+1)$ times yields the assertion for $Z_q$.
\end{proof}

\begin{lem}
Let $I$ be the category of finite sets and injective functions,
and let $\bfm \in I$.  If $\C \: I \to \perm$ is a functor that
sends each $\vphi \: \bfm \to \bfn \in I$ to a stable
(resp.~unstable) equivalence $\C(\vphi) \: \C(\bfm) \to \C(\bfn)$,
then the canonical functors
$$
\C(\bfm) \ew Z\C(\bfm) \lra \Dhocolim_I \C
$$
are stable (resp.~unstable) equivalences.
\end{lem}

\begin{proof}
This follows from Lemmas~\ref{lem:boekstedt_lemma}
and ~\ref{lem:epsunstable}.
\end{proof}

%% [[JR: I removed this remark, because it seems imprecise.]]
%% 
%% \begin{rem} \label{ex:whathappensforisolatedzeros}
%% In the situation where $\N$ is a permutative category without zero,
%% consider the permutative category $\M = L\N$.  Then, for each $\bfn \in I$
%% the diagram $T \mapsto G\M(\bfn,T)$ is a diagram in $\strictnz$ with a
%% disjoint zero added, and so the based homotopy colimit over $\Q\bfn$
%% would be appropriate.  However, once $\bfn$ starts moving, the zero
%% is used actively, and the diagram $G\M \: I\smallint\Q \to \strict$ is in
%% $\strictiz$ only, so here we will need the derived homotopy colimit.
%% \end{rem}
%% 
%% [[JR: In $0 \gets \M \to \M \times \M$, the functor $\M \to 0$ is not of
%% the form $LF$, so I don't think that the diagram $T \mapsto G\M(\bfn,T)$
%% is a diagram in $\strictnz$ with a disjoint zero added.  The objects
%% are of this kind, but not the morphisms.]]

\section{The homotopy colimit of a graded bipermutative category}
\label{sec:hocolimgradedbps}

We are now ready for a key proposition.

\begin{prop}\label{prop:hocolim of rigs}
Let $J$ be a permutative category, and let $\C^\bullet$ be a $J$-graded
\bp.  Then $\Dhocolim_J\C^\bullet$ is a simplicial \bp, and
$$
\C^0 \ew Z\C^0 \lra \Dhocolim_J \C^\bullet
$$
are lax morphisms of simplicial \bps.  The same statements hold when
replacing ``bipermutative'' by ``strictly bimonoidal''.

Furthermore, for each $x\in J$, the canonical functors
$$
\C^x \ew Z\C^x \lra \Dhocolim_J \C^\bullet
$$
are maps of $Z\C^0$-modules.
\end{prop}

\begin{proof}
Recall the adjoint pair $(L',R')$ from subsection~\ref{sec:withzero}.
If $\C^\bullet$ is a $J$-graded $\bp$, then so is $L'R'\C^\bullet =
\C^\bullet_+$, and $Z\C^\bullet$ becomes a simplicial $J$-graded $\bp$.
By Lemma~\ref{lem:zeroless}, which we will prove below, we get that
$\hocolim_J R'(L'R')^q\C^\bullet$ becomes a zeroless \bp for each $q\geq0$.
Hence $\hocolimiz_J Z_q\C^\bullet = L\hocolim_J R'(L'R')^q\C^\bullet$ is a
\bp, and all the simplicial structure maps are lax morphisms of \bps.
Therefore $\Dhocolim_J \C^\bullet$ becomes a simplicial $\bp$.

Likewise, for each $q\geq0$, Lemma~\ref{lem:zeroless} below guarantees that
$$
Z_q\C^0\to \hocolimiz_J Z_q\C^\bullet
$$
is a lax morphism of \bps and that each
$$
Z_q\C^x\to \hocolimiz_J Z_q\C^\bullet
$$
is a map of $Z_q\C^0$-modules, so we are done by functoriality.
\end{proof}

We omit the forgetful functors $R$ and $R'$ in the statement and
proof of the following lemma, which contains the most detailed
diagram chasing required in this paper.

\begin{lem}\label{lem:zeroless}\label{lem:zerolesssym}
Let $J$ be a permutative category.  If $\C^\bullet$ is a $J$-graded
\bp, then Thomason's homotopy colimit of permutative categories
$\hocolim_J \C^\bullet$ is a zeroless \bp.  The canonical functor
$\C^0 \to \hocolim_J \C^\bullet$ is a lax morphism of zeroless \bps.
Furthermore, for each $x\in J$, the canonical functor
$$
\C^x \lra \hocolim_J \C^\bullet
$$
is a map of zeroless $\C^0$-modules.

If $\C^\bullet$ is a $J$-graded \sbc, then $\hocolim_J \C^\bullet$
is a zeroless \sbc with a lax morphism of zeroless \sbcs $\C^0 \to
\hocolim_J \C^\bullet$, and zeroless $\C^0$-module maps $\C^x \to
\hocolim_J \C^\bullet$.
\end{lem}

\begin{proof}
Thomason showed that the homotopy colimit is a permutative category
without zero.  The additive twist isomorphism
\begin{multline*}
\tau_\oplus \: n[(x_1,X_1), \ldots, (x_n,X_n)]
	\oplus m[(y_1,Y_1), \ldots, (y_m,Y_m)] \\
\overset{\cong}\lra m[(y_1,Y_1), \ldots, (y_m,Y_m)]
	\oplus n[(x_1,X_1), \ldots, (x_n,X_n)]
\end{multline*}
is given by $(\chi(n,m),\id,\id)$, where $\chi(n,m) \in \Sigma_{n+m}$
is as in Example~\ref{ex:bipsets}.  Let $[X]$ and $[Y]$ be shorthand
notations for the objects $n[(x_1,X_1), \ldots, (x_n,X_n)]$ and
$m[(y_1,Y_1), \ldots, (y_m,Y_m)]$, respectively.  The twist isomorphism for
$\oplus$ then appears as
$$
\tau_\oplus \: [X] \oplus [Y] \overset{\cong}\lra [Y] \oplus [X] \,.
$$

In order to distinguish the multiplicative structure of $\C^\bullet$ from
the one on the homotopy colimit, we shall simply denote the composition
functor $\otimes$ on $\C^\bullet$ by juxtaposition of objects, or
by $\cdot$.  The multiplicative bifunctor $\otimes$ on the homotopy
colimit is then defined at the object level by
\begin{multline*}
n[(x_1,X_1),\ldots,(x_n,X_n)] \otimes m[(y_1,Y_1), \ldots,(y_m,Y_m)] \\
:= nm[(x_1+y_1,X_1Y_1), \ldots, (x_1+y_m,X_1Y_m),
	\ldots, (x_n+y_1, X_nY_1), \ldots, (x_n+y_m, X_nY_m)] \,.
\end{multline*}
We will use the shorthand notation $[X] \otimes [Y]$ for this object.

The object $1 := 1[(0,1)]$ is a unit for $\otimes$.  With these
definitions, as extended below to the morphism level, $(\hocolim_J
\C^\bullet, \otimes, 1)$ is a strict monoidal category.

We will define the multiplicative twist map $\tau_\otimes \: [X] \otimes
[Y] \overset{\cong}\lra [Y] \otimes [X]$ as a composite of two morphisms.
First, we apply the twist map $\gamma_\otimes$ for
the multiplication in $\C^\bullet$ to every entry of the form $X_iY_j$.
The triple $(\id_{\mathbf{nm}}, \chi^{x_i,y_j}, \gamma_\otimes)$ defines
a morphism
\begin{multline*}
nm[(x_1+y_1,X_1Y_1), \ldots, (x_1+y_m,X_1Y_m), \ldots, (x_n+y_1,
X_nY_1), \ldots, (x_n+y_m, X_nY_m)] \\
\to
nm[(y_1+x_1,Y_1X_1), \ldots, (y_m+x_1,Y_mX_1),
	\ldots, (y_1+x_n, Y_1X_n), \ldots, (y_m+x_n, Y_mX_n)] \,.
\end{multline*}
(Here $\chi$ is the twist map in $J$.  To be precise, the third coordinate
of the morphism is really $\C(\chi^{x_i,y_j})(\gamma_\otimes)$, but we
omit $\C(\chi^{x_i,y_j})$ from the notation.)
Second, we use the permutation $\sigma_{n,m} \in \Sigma_{nm}$ that induces
matrix transposition.  The triple $(\sigma_{n,m}, \id_{y_j+x_i}, \id)$
defines a morphism
\begin{multline*}
nm[(y_1+x_1,Y_1X_1), \ldots, (y_m+x_1,Y_mX_1),
	\ldots, (y_1+x_n, Y_1X_n), \ldots, (y_m+x_n, Y_mX_n)] \\
\to
mn[(y_1+x_1,Y_1X_1), \dots, (y_1+x_n,Y_1X_n), \dots,
	(y_m+x_1,Y_mX_1), \dots, (y_m+x_n,Y_mX_n)] \,.
\end{multline*}
Let the twist map for $\otimes$ be the composite morphism $\tau_\otimes
= (\sigma_{n,m}, \id_{y_j+x_i}, \id) \circ (\id_{\mathbf{nm}},
\chi^{x_i,y_j}, \gamma_\otimes)$.

As matrix transposition squares to the identity, $\chi^{y_j,x_i}
\circ \chi^{x_i,y_j} = \id$ and
$\gamma_\otimes^2 = \id$, we obtain that $\tau_\otimes^2 = \id$. %%%3110
If $[X] = 1$ is the multiplicative unit, then we have that $\sigma_{1,m}$ is
the identity in $\Sigma_m$ and $\chi^{0,y_j}$ is the identity as
well, so $\tau_\otimes \: 1 \otimes [Y] \to [Y] \otimes 1$ is
the identity.  Similarly one shows that $\tau_\otimes$ gives the identity
morphism if $[Y] = 1$ is the multiplicative unit.

We have now verified properties~\eqref{igradedtensor}, \eqref{igradedunit}
and~\eqref{igradedmulttwist} of Definition~\ref{def:igradedbim},
at the level of objects.  We leave to the reader to check
property~\eqref{igradedmult}.  Property~\eqref{igradedzeros} is
disregarded in the zeroless situation.

Writing out $([X]\otimes [Y])\oplus([X']\otimes [Y])$ and
$([X]\oplus[X'])\otimes Y$ we get the same object, and we define the
\rightdistributivity $\dr$ to be the identity morphism between these
two expressions.  The \leftdistributivity $\dl$ involves a reordering of
elements.  It is a morphism
$$
\dl \: ([X]\otimes[Y]) \oplus ([X]\otimes [Y'])
	\lra [X]\otimes ([Y] \oplus [Y']) \,.
$$
The source is
$$
(nm+nm')[(x_1+y_1,X_1Y_1), \ldots, (x_n+y_m, X_nY_m),
	(x_1+y'_1, X_1Y'_1), \ldots, (x_n+y'_{m'}, X_nY'_{m'})] \,,
$$
while the target is
$$
n(m+m')[(x_1+y_1, X_1Y_1), \ldots, (x_1+y'_{m'}, X_1Y'_{m'}),
	\ldots, (x_n+y_1, X_nY_1), \ldots, (x_n+y'_{m'}, X_nY'_{m'})] \,.
$$
The same terms $(x_i+y_j,X_iY_j)$ and $(x_i+y'_j,X_iY'_j)$ occur in
both the source and target, but their ordering differs by a suitable
permutation $\xi \in \Sigma_{nm+nm'}$.  Thus we define the morphism
$\dl$ by the triple $(\xi,\id,\id)$.  Note that $\xi$ is the
\leftdistributivity isomorphism in the bipermutative category of finite
sets and functions, as defined in Example~\ref{ex:bipsets}.

We have to check that the so defined distributivity transformation $\dl$
coincides with $\tau_\otimes \circ (\tau_\otimes \oplus \tau_\otimes)$.
The twist terms $\gamma_\otimes$ and $\chi$ occur twice in the
composition, so they reduce to the identity.  What is left is a
permutation that is caused by $\tau_\otimes \circ (\tau_\otimes \oplus
\tau_\otimes)$, and this is precisely $\xi$.

We have now verified properties~\eqref{igradedrightdist}
and~\eqref{igradedleftdist} of Definition~\ref{def:igradedbim}.
Since the isomorphisms $\tau_\oplus$, $\dr$ and $\dl$ are all of
the form $(\sigma, \id, \id)$ for suitable permutations $\sigma$,
properties~\eqref{igradedpermutation}, \eqref{igradeddistass}
and~\eqref{igradedpentagon} all follow from the corresponding ones in
the bipermutative category of finite sets and functions.

\medskip 

This finishes the proof that the zeroless \bp structure works fine
on objects.  It remains to establish that $\oplus$ and $\otimes$ are
bifunctors on $\hocolim_J \C^\bullet$, that the various associativity and
distributivity laws are natural, and that the additive and multiplicative
twists are natural.

\medskip 

For $\oplus$ this is straightforward
and can be found in \cite{Th3}: suppose given two morphisms
$$
(\psi, \ell_i, \vrho_j) \: n[(x_1,X_1),\ldots,(x_n,X_n)]
\to n'[(x'_1,X'_1),\ldots,(x'_{n'},X'_{n'})]
$$
and
$$
(\vphi, k_i, \pi_j) \: m[(y_1,Y_1),\ldots,(y_m,Y_m)] \to
m'[(y'_1,Y'_1),\ldots,(y'_{m'},Y'_{m'})]
$$
in the homotopy colimit,
with $\psi\: \bfn \to \bfn'$, $\ell_i: x_i
\rightarrow x'_{\psi(i)}$ and $\vrho_j: \bigoplus_{\psi(i)=j}
\C(\ell_i)(X_i) \rightarrow X'_j$, and $\vphi\: \bfm \to
\bfm'$ with corresponding $k_i$ and $\pi_j$.
Then there is a surjection $\psi + \vphi$ from $\mathbf{n+m}$ to
$\mathbf{n'+m'}$, and we can recycle the morphisms $\ell_i$ and $k_i$ to
give corresponding morphisms in $J$.  In the third coordinate
we can use the morphisms $\vrho_j$ and $\pi_j$
to get new ones, because the preimages of $\bfn'$ and
$\bfm'$ under $\psi + \vphi$ are disjoint.  Taken together, this
results in a
morphism from the sum $(n+m)[(x_1,X_1), \ldots, (y_m, Y_m)]$ to the
sum $(n'+m')[(x'_1,X'_1), \ldots, (y'_{m'}, Y'_{m'})]$.  It is
straightforward to see that $\oplus$ defines a bifunctor, that the
associativity law for $\oplus$ is natural, and
that the additive twist $\tau_\oplus$ is natural.

%%%%%%%%%%%%%%%%%%%%otimes on morphisms -- new version

For the remainder of this proof let us denote the elements in the set
$\mathbf{nm}=\{1, \ldots, nm\}$ as pairs $(i,j)$ with $1\leq i \leq n$
and $1 \leq j \leq m$.
The tensor product of the morphisms $(\psi, \ell_i, \vrho_j)$ and $(\vphi,
k_i, \pi_j)$ has three coordinates.  On the first we take the product
of the surjections, \ie,
$$
 \bfn \bfm \ni (i,j) \mapsto (\psi(i),\vphi(j)) \in \bfn'\bfm',
$$
and on the second we take the sum $\ell_i + k_j\: x_i +y_j \to
x'_{\psi(i)} + y'_{\vphi(j)}$ of the morphisms $\ell_i, k_j \in J$.
The third coordinate of the morphism
$(\psi, \ell_i, \vrho_j) \otimes (\vphi, k_i, \pi_j)$
has to be a morphism
$$
\bigoplus_{(\psi(i), \vphi(j))=(r,s)} \C(\ell_i +
k_j)(X_i \cdot Y_j) = \bigoplus_{(\psi(i), \vphi(j))=(r,s)}
\C(\ell_i)(X_i) \cdot \C(k_j)(Y_j) \lra X_r'\cdot Y_s'
$$
in $\C(x'_r+y'_s)$, for each $1 \leq r \leq n'$ and $1 \leq s \leq m'$.
Here, the sum is taken with respect to the lexicographical ordering of
the indices $(i,j)$.  Consider the following diagram:
$$
\xymatrix@C=-1.5cm{
  & \bigoplus_{(\psi(i),\vphi(j))=(r,s)}
  \C(\ell_i)(X_i) \cdot \C(k_j)(Y_j) \ar[dl]_{\id} \ar[dr]^{\sigma} \\
\bigoplus_{\psi(i)=r} \bigoplus_{\vphi(j)=s}
  \C(\ell_i)(X_i) \cdot \C(k_j)(Y_j) \ar[d]_{\bigoplus_{\psi(i)=r} \dl} & &
\bigoplus_{\vphi(j)=s} \bigoplus_{\psi(i)=r}
  \C(\ell_i)(X_i) \cdot \C(k_j)(Y_j) \ar[d]^{\bigoplus_{\vphi(j)=s} \dr} \\
\bigoplus_{\psi(i)=r} \C(\ell_i)(X_i) \cdot
  \left( \bigoplus_{\vphi(j)=s} \C(k_j)(Y_j) \right)
  \ar[d]_{\bigoplus_{\psi(i)=r} \id \cdot \pi_s} & &
\bigoplus_{\vphi(j)=s}
  \left( \bigoplus_{\psi(i)=r} \C(\ell_i)(X_i) \right) \cdot \C(k_j)(Y_j)
  \ar[d]^{\bigoplus_{\vphi(j)=s} \vrho_r \cdot \id} \\
\bigoplus_{\psi(i)=r} \C(\ell_i)(X_i) \cdot Y_s' \ar[d]_{\dr} & &
\bigoplus_{\vphi(j)=s} X_r' \cdot \C(k_j)(Y_j) \ar[d]^{\dl} \\
\left( \bigoplus_{\psi(i)=r} \C(\ell_i)(X_i) \right) \cdot Y_s'
  \ar[rd]_{\vrho_r \cdot \id} & &
X'_r \cdot \left( \bigoplus_{\vphi(j)=s} \C(k_j)(Y_j) \right) 
  \ar[ld]^{\id \cdot \pi_s} \\
  & X_r' \cdot Y_s'
}
$$
The isomorphism $\sigma$ is an appropriate permutation of the summands.
The distributivity laws in $\C^\bullet$ are natural with respect to
morphisms in $\C^\bullet$, and therefore we have the identities:
\begin{align*}
\dr \circ \Bigl( \bigoplus_{\psi(i)=r} \id_{\C(\ell_i)(X_i)} \cdot \pi_s \Bigr)
&= \left( \left( \id_{\bigoplus_{\psi(i)=r} \C(\ell_i)(X_i)} \right)
  \cdot \pi_s \right) \circ \dr \\
\dl \circ \Bigl( \bigoplus_{\vphi(j)=s} \vrho_r \cdot \id_{\C(k_j)(Y_j)} \Bigr)
&= \left(\vrho_r \cdot
  \left(\id_{\bigoplus_{\vphi(j)=s} \C(k_j)(Y_j)} \right) \right) \circ \dl
\end{align*}
Combining these with the generalized pentagon equation
$$
\dr \circ \bigoplus_{\psi(i)=r} \dl
	= \dl \circ \bigoplus_{\vphi(j)=s} \dr \circ \sigma
$$
we see that the diagram commutes.  We define the third coordinate in
the tensor product morphism to be the composition given by either of
the two branches.

Note that for $(\psi, \ell_i, \vrho_j) \otimes \id$ the definition
reduces to $(\vrho_j \cdot \id) \circ \dr$, and similarly the third
coordinate of $\id \otimes (\vphi, k_i, \pi_j)$ is $(\id \cdot \pi_j)
\circ \dl$.
In particular, the tensor product of identity morphisms is an identity
morphism.

Compositions of morphisms in the homotopy colimit involve an additive
twist \cite{Th3}*{p.~1631}.  For
$$
(\psi', \ell'_i, \vrho'_j)\:
n'[(x'_1,X'_1),\ldots,(x'_{n'},X'_{n'})] \rightarrow
n''[(x''_1,X''_1),\ldots,(x''_{n''},X''_{n''})]
$$
the morphism $\bigoplus_{\psi'\psi(i)=r} \C(\ell'_{\psi(i)}\ell_i)(X_i)
\to X_r''$ is given as a composition.  First, one has to permute the
summands
$$
\sigma \: \bigoplus_{\psi'\psi(i)=r} \C(\ell'_{\psi(i)}\ell_i)(X_i) \to
\bigoplus_{\psi'(k)=r} \ \bigoplus_{\psi(i)=k}
\C(\ell'_k \ell_i)(X_i) \,.
$$
Then, as we assumed that $\C$ is a functor to $\strict$, we know that
$$
\bigoplus_{\psi'(k)=r} \ \bigoplus_{\psi(i)=k} \C(\ell'_k\ell_i)(X_i)
  = \bigoplus_{\psi'(k)=r} \ \bigoplus_{\psi(i)=k} \C(\ell'_k)\C(\ell_i)(X_i)
  = \bigoplus_{\psi'(k)=r} \C(\ell'_k)
\Bigl( \bigoplus_{\psi(i)=k} \C(\ell_i)(X_i) \Bigr) \,.
$$
Finally, we apply the morphism
$$
\bigoplus_{\psi'(k)=r} \C(\ell'_k)(\vrho_k) \:
\bigoplus_{\psi'(k)=r} \C(\ell'_k)
  \Bigl( \bigoplus_{\psi(i)=k} \C(\ell_i)(X_i) \Bigr)
\lra \bigoplus_{\psi'(k)=r} \C(\ell'_k)(X'_k)
$$
and continue with $\vrho'_r$ to end up in $X''_r$.

In order to prove that the tensor product actually defines a bifunctor,
we will show that
$$
 (\psi,\ell_i,\vrho_j) \otimes (\vphi,k_i,\pi_j)
= ((\psi,\ell_i,\vrho_j) \otimes \id)
	\circ (\id \otimes (\vphi,k_i,\pi_j))
= (\id \otimes (\vphi,k_i,\pi_j))
	\circ ((\psi,\ell_i,\vrho_j) \otimes \id)
$$
and
$$
((\psi',\ell'_i,\vrho'_j) \otimes \id)
	\circ ((\psi,\ell_i,\vrho_j) \otimes \id)
= ((\psi',\ell'_i,\vrho'_j) \circ (\psi,\ell_i,\vrho_j)) \otimes \id \,,
$$
and leave the check of the remaining identity to the reader.

The first equation is straightforward to see, because
$((\psi,\ell_i,\vrho_j) \otimes \id) \circ (\id \otimes
(\vphi,k_i,\pi_j))$ corresponds to the left branch of the diagram above
and the other composition is given by the right branch.

For the second equation we have to check that
$(((\vrho' \circ \vrho) \cdot \id) \circ \dr)_s =
((\vrho'\cdot \id) \circ \dr \circ (\vrho \cdot \id) \circ \dr)_s$.
Both morphisms have source
$$
 \bigoplus_{\psi'\psi(i)=s} \C(\ell'_{\psi(i)} \ell_i +\id)(X_i \cdot Y_j) =
\bigoplus_{\psi'\psi(i)=s} \C(\ell'_{\psi(i)} \ell_i)(X_i)\cdot Y_j
$$
and the left hand side corresponds to the left branch of the
following diagram and the right hand side to the right branch.
$$
\xymatrix@C=-0.5cm{
{\bigoplus_{\psi'\psi(i)=s} \C(\ell'_{\psi(i)}
  \ell_i)(X_i)\cdot Y_j} \ar[d]_{\dr}
 \ar@{=}[rr]&{}&{\bigoplus_{\psi'\psi(i)=s} \C(\ell'_{\psi(i)}
   \ell_i + \id)(X_i \cdot
   Y_j)}\ar[d]^{\sigma}  \\
{\left(\bigoplus_{\psi'\psi(i)=s} \C(\ell'_{\psi(i)}
  \ell_i)(X_i)\right)\cdot Y_j} \ar[d]_{\sigma \cdot \id}
&{}&{\bigoplus_{\psi'(k)=s}\bigoplus_{\psi(i)=k}
    \C(\ell'_{\psi(i)}+\id)\C(\ell_i+\id)(X_i\cdot Y_j)}
\ar@{=}[d]\\
{\left(\bigoplus_{\psi'(k)=s}\bigoplus_{\psi(i)=k}
    \C(\ell'_{\psi(i)})\C(\ell_i)(X_i)\right)\cdot Y_j}
\ar@{=}[d]
&{}&
{\bigoplus_{\psi'(k)=s} \C(\ell'_k+\id)\left(\bigoplus_{\psi(i)=k}
      \C(\ell_i + \id)(X_i \cdot
    Y_j)\right)}\ar[d]^{\bigoplus_{\psi'(k)=s}
    \C(\ell'_k+\id)((\vrho_k\cdot \id)\circ \dr)}\\
{\left(\bigoplus_{\psi'(k)=s} \C(\ell'_k)\left(\bigoplus_{\psi(i)=k}
\C(\ell_i)(X_i)\right)\right)\cdot Y_j}
\ar[d]_{\left(\bigoplus_{\psi'(k)=s} \C(\ell'_k)(\vrho_k)\right) \cdot \id}
&{}&
{\bigoplus_{\psi'(k)=s} \C(\ell'_k)(X'_k)\cdot Y_j}\ar[d]^{\dr} \\
{\left(\bigoplus_{\psi'(k)=s} \C(\ell'_k)(X'_k)\right) \cdot
  Y_j} \ar[dr]_{\vrho'_s \cdot \id} \ar@{=}[rr]
&{}&{\left(\bigoplus_{\psi'(k)=s} \C(\ell'_k)(X'_k)\right) \cdot
  Y_j} \ar[ld]^{\vrho'_s \cdot \id}\\
{}&{X''_s\cdot Y_j}&{}
}$$
Naturality of $\dr$ in $\C^\bullet$ ensures that $\dr$ can change place
with $\bigoplus_{\psi'(k)=s}\C(\ell'_k+\id)(\vrho_k\cdot \id)$ on the
right branch.  That $\dr \circ \sigma = (\sigma \cdot \id) \circ \dr$
holds because $\C^\bullet$ satisfies property~\eqref{igradedpermutation}
from Definition~\ref{def:igradedbim}, and hence the diagram commutes.

\medskip

In order to show that the associativity identification is natural, we
have to prove that
$$
((\psi^1,\ell_i^1,\vrho_j^1) \otimes
(\psi^2,\ell_i^2,\vrho_j^2)) \otimes (\psi^3,\ell_i^3,\vrho_j^3) =
(\psi^1,\ell_i^1,\vrho_j^1) \otimes
((\psi^2,\ell_i^2,\vrho_j^2) \otimes
(\psi^3,\ell_i^3,\vrho_j^3))
$$
for morphisms in the homotopy colimit.  The claim is obvious on the
coordinates of the surjections and the morphisms in $J$.

For proving the identity in the third coordinate of morphisms, note
that the naturality of $\otimes$ implies that we can write
\begin{align*}
& ((\psi^1,\ell_i^1,\vrho_j^1) \otimes
(\psi^2,\ell_i^2,\vrho_j^2)) \otimes (\psi^3,\ell_i^3,\vrho_j^3) \\
= & (((\psi^1,\ell_i^1,\vrho_j^1) \otimes \id) \otimes \id) \circ
((\id \otimes (\psi^2,\ell_i^2,\vrho_j^2)) \otimes \id) \circ
((\id \otimes \id) \otimes (\psi^3,\ell_i^3,\vrho_j^3)) \,.
\end{align*}
Therefore, it suffices to prove the claim for each of the factors.  We
will show it for the middle one and leave the other ones to the
curious reader.  Recall that $\id \otimes
(\psi^2,\ell_i^2,\vrho_j^2)$ has as third coordinate the composition
$(\id \cdot \vrho_j^2) \circ \dl$ and therefore
$(\id \otimes (\psi^2,\ell_i^2,\vrho_j^2)) \otimes \id$ has third
coordinate
$$
 (((\id \cdot \vrho_j^2) \circ \dl) \cdot \id) \circ \dr = (\id
\cdot \vrho_j^2 \cdot \id) \circ (\id \cdot \dl) \circ \dr \,.
$$
But $(\id \cdot \dl) \circ \dr = (\dr \cdot \id) \circ \dl$
(equation (\ref{igradedleftdist}') of Definition~\ref{def:igradedbc}) holds in $\C^\bullet$, and therefore the
third coordinate equals
$$
(\id
\cdot \vrho_j^2 \cdot \id) \circ (\id \cdot \dl) \circ \dr
= (\id
\cdot \vrho_j^2 \cdot \id) \circ (\dr \cdot \id) \circ \dl
$$
which is the third coordinate of $\id \otimes
((\psi^2,\ell_i^2,\vrho_j^2) \otimes \id)$.

\medskip

Naturality of the multiplicative twist map can be seen as follows.
We have to show that
$$
 \tau_\otimes \circ ((\psi, \ell_i, \vrho_j) \otimes (\vphi, k_i,
\pi_j)) = ((\psi, \ell_i, \vrho_j) \otimes (\vphi,
k_i, \pi_j)) \circ \tau_\otimes \,.
$$
On the first coordinate of the morphisms this reduces to the equality
$$
 \sigma_{n',m'} \circ (\psi,\vphi)(i,j) = (\vphi(j),\psi(i)) =
(\vphi, \psi) \circ \sigma_{n,m}(i,j),
$$
and on the second coordinate we have the equation
$$
 \chi \circ (\ell_i + k_j) = (k_j + \ell_i) \circ \chi
$$
because $\chi$ is natural.  Thus, it remains to prove that the above
equation holds in the third coordinate, which amounts to showing that the
following diagram commutes.
$$
\xymatrix{
{\bigoplus_{\psi(i)=r} \bigoplus_{\vphi(j)=s} \C(\ell_i)(X_i)
\cdot \C(k_j)(Y_j)} \ar[rr]^{(\bigoplus\bigoplus \gamma_\otimes)
\circ \sigma} \ar[dd]_{\bigoplus_{\psi(i)=r}\dl} & {} &
{\bigoplus_{\vphi(j)=s}
\bigoplus_{\psi(i)=r} \C(k_j)(Y_j)
\cdot \C(\ell_i)(X_i)}\ar[d]^{\sigma^{-1}}\\
{}&{}&{\bigoplus_{\psi(i)=r} \bigoplus_{\vphi(j)=s}
\C(k_j)(Y_j) \cdot \C(\ell_i)(X_i)}
\ar[d]^{\bigoplus_{\psi(i)=r} \dr}\\
{\bigoplus_{\psi(i)=r} \C(\ell_i)(X_i)
\cdot \left(\bigoplus_{\vphi(j)=s}\C(k_j)(Y_j)\right)}
\ar[d]_{\bigoplus_{\psi(i)=r} \id \cdot \pi_s}
\ar@{.>}[rr]^{\bigoplus_{\psi(i)=r}
\gamma_\otimes}&{}&
{\bigoplus_{\psi(i)=r}  \left(\bigoplus_{\vphi(j)=s}
\C(k_j)(Y_j)\right) \cdot \C(\ell_i)(X_i)}
\ar[d]^{\bigoplus_{\psi(i)=r} \pi_s \cdot \id}\\
{\bigoplus_{\psi(i)=r} \C(\ell_i)(X_i)
\cdot Y'_s} \ar[d]_{\dr}&{}&{\bigoplus_{\psi(i)=r}Y'_s \cdot
\C(\ell_i)(X_i)}\ar[d]^{\dl}\\
{\left(\bigoplus_{\psi(i)=r}\C(\ell_i)(X_i)\right)\cdot
Y'_s} \ar[d]_{\vrho_r \cdot \id}&{}&{Y'_s \cdot
\left(\bigoplus_{\psi(i)=r}\C(\ell_i)(X_i)\right)}\ar[d]^{\id \cdot
\vrho_r}\\
{X'_r \cdot Y'_s} \ar[rr]^{\gamma_\otimes}&{}&{Y'_s \cdot X'_r}
}
$$
The top diagram commutes because $\dl$ is defined in terms of $\dr$ and
$\gamma_\otimes$.  For the bottom diagram we apply the same argument
together with the naturality of $\gamma_\otimes$.

\medskip

We have to check that \rightdistributivity is the
identity on morphisms.
Consider three morphisms as above.  When we focus on the surjections
$\psi^1\: \bfn \rightarrow \bfn'$, $\psi^2\: \bfm \rightarrow \bfm'$, and
$\psi^3\: \bfl \rightarrow \bfl'$, we see that a condition like
$(\psi^1 + \psi^2)\psi^3(i,j) = (r,s)$ only affects either the preimage of
$\bfn'\bfl'$ or the preimage of $\bfm'\bfl'$ in $(\bfn+\bfm)\bfl$, but never
both.  Therefore, the third coordinate of the morphism
$$
 ((\psi^1,\ell_i^1,\vrho_j^1) \oplus
(\psi^2,\ell_i^2,\vrho_j^2)) \otimes (\psi^3,\ell_i^3,\vrho_j^3)
$$
is either a third coordinate of
$ (\psi^1,\ell_i^1,\vrho_j^1)\otimes (\psi^3,\ell_i^3,\vrho_j^3)$
or of $ (\psi^2,\ell_i^2,\vrho_j^2)\otimes
(\psi^3,\ell_i^3,\vrho_j^3)$, and thus \rightdistributivity is the
identity on morphisms.

%%%%%%3110
In the $J$-graded bipermutative case the naturality of the
\leftdistributivity isomorphism follows from the one of $\dr$ and the
multiplicative twist.  In both the bipermutative and the strictly
bimonoidal case \leftdistributivity is given by $(\xi, \id, \id)$.
Therefore naturality of $\dl$ in the bipermutative setting proves
naturality in the strictly bimonoidal setting as well.
%%%%%%%3110

\medskip

This finishes the proof that $\hocolim_J \C^\bullet$ is a bipermutative
category without zero.  We now prove the remaining statements of the
lemma.

\medskip

There is a natural functor $G \: \C^0 \to \hocolim_J \C^\bullet$ which
sends $X \in \C^0$ to $G(X) = 1[(0,X)]$.  Note that the functor $G$ is
strict (symmetric) monoidal with respect to $\otimes$, because $G(1) =
1[(0,1)]$ and
$$
G(X) \otimes G(Y) = 1[(0,X)] \otimes 1[(0,Y)] = 1[(0+0, X \otimes Y)]
= 1[(0, X \otimes Y)] = G(X \otimes Y) \,.
$$
However, $G$ is only lax symmetric monoidal with respect to $\oplus$:
there is a binatural transformation $\eta_\oplus = (\psi, \id, \id)$
from $G(X) \oplus G(X') = 1[(0,X)] \oplus 1[(0,X')] = 2[(0,X),(0,X')]$ to
$G(X \oplus X') = 1[(0, X\oplus X')]$, given by the canonical surjection
$\psi \: \mathbf2 \to \mathbf1$ and identity morphisms in the other two
components.  This morphism is of course not an isomorphism.

We have to show that the functor $G$ respects the distributivity
constraints $\dr = \id$ and $\dl$.  In our situation we have that
$\eta_\otimes = \id$, so we have to check that
$$
\eta_\oplus = \eta_\oplus \otimes \id
$$
and
$$
G(\tau_\otimes \circ (\tau_\otimes \oplus \tau_\otimes))
  \circ \eta_\oplus = (\id \otimes \eta_\oplus)
  \circ \tau_\otimes \circ (\tau_\otimes \oplus \tau_\otimes)
\,.
$$
The first equation is just stating the fact that
$$
\xymatrix{
2[(0,X), (0,X')] \otimes 1[(0,Y)] \ar[rr]^-{\eta_\oplus \otimes \id}
  \ar@{=}[d] && 1[(0, X \oplus X')] \otimes 1[(0, Y)] \ar@{=}[d] \\
2[(0,X \otimes Y), (0,X' \otimes Y)] \ar[rr]^-{\eta_\oplus}
  && 1[(0, (X \oplus X') \otimes Y)]
}
$$
commutes, in view of the identity $\dr \: (X \otimes Y) \oplus (X' \otimes Y)
= (X \oplus X') \otimes Y$.

For the \leftdistributivity law we should observe that the
multiplicative twist $\tau_\otimes$ on the homotopy colimit reduces to
the morphism $(\id,\chi,\gamma_\otimes)$ in the case of elements of
length $1$ in the homotopy colimit, and that $\chi^{0,0} = \id$.
Furthermore, $\id \otimes (\psi,
\id,\id) = (\psi, \id, \id)$ holds.  Therefore
\begin{align*}
(\id \otimes \eta_\oplus) \circ \dl
&= (\id \otimes (\psi, \id,\id))
\circ \tau_\otimes \circ (\tau_\otimes \oplus \tau_\otimes) \\
&= (\psi, \id,\id)
\circ (\id,\id,\gamma_\otimes \circ (\gamma_\otimes \oplus \gamma_\otimes)) \\
&=
(\id,\id,\gamma_\otimes \circ (\gamma_\otimes \oplus \gamma_\otimes))
\circ (\psi, \id,\id) = G(\dl) \circ \eta_\oplus \,.
\end{align*}
The claim about the module structure is obvious.

As the \leftdistributivity on the homotopy colimit is of the form
$(\xi,\id,\id)$, the above proof carries over to the strictly
bimonoidal case.
\end{proof}

\begin{lem} \label{lem:indmaps}
If $F \: \C^\bullet \rightarrow \D^\bullet$ is a lax morphism of
$J$-graded \bps (resp.~$J$-graded \sbcs) then it induces a lax morphism
$F_* \: \hocolim_J \C^\bullet \rightarrow \hocolim_J \D^\bullet$ of
zeroless \bps (resp.~zeroless \sbcs).
\end{lem}

\begin{proof}
Of course, we define $F_*\: \hocolim_J \C^\bullet \to \hocolim_J
\D^\bullet$ on objects by
$$
F_*(n[(x_1,X_1),\ldots, (x_n,X_n)])
:= n[(x_1,F(X_1)),\ldots, (x_n,F(X_n))] \,.
$$
Note that with this definition $F_*$ is \emph{strict} symmetric monoidal
with respect to $\oplus$ even if $F$ was only lax symmetric monoidal.

Given a morphism $(\psi,\ell_i,\vrho_j)$ from $n[(x_1,X_1), \ldots,
(x_n,X_n)]$ to $m[(y_1,Y_1), \ldots, (y_m,Y_m)]$ we define the induced
morphism
$$
(\psi,\ell_i,\vrho_j^F) \:
F_*(n[(x_1,X_1), \ldots, (x_n,X_n)]) \to F_*(m[(y_1,Y_1), \ldots,
(y_m,Y_m)])
$$
as follows: we keep the surjection $\psi$ and the morphisms $\ell_i$,
and for the third coordinate we take the composition
$$
\vrho_j^F \: \bigoplus_{\psi(i)=j} \D(\ell_i)(F(X_i))
= \bigoplus_{\psi(i)=j} F( \C(\ell_i)(X_i) )
  \overset{\eta_\oplus}\lra
F(\bigoplus_{\psi(i)=j} \C(\ell_i)(X_i) )
  \overset{F(\vrho_j)}\lra F(Y_j) \,.
$$
The naturality of $\eta_\oplus$ ensures that composition of morphisms
is well-defined.

Let $n[(x_1,X_1),\ldots, (x_n,X_n)]$ and $m[(y_1,Y_1),\ldots, (y_m,Y_m)]$
be two objects in $\hocolim_J \C^\bullet$.  Applying $\otimes \circ
(F_*,F_*)$ yields
$$
nm[(x_1+y_1,F(X_1)\otimes F(Y_1)),\ldots,(x_n+y_m, F(X_n) \otimes F(Y_m))]
$$
whereas the composition $F_* \circ \otimes$ gives
$$
nm[(x_1+y_1,F(X_1\otimes Y_1)),\ldots,(x_n+y_m, F(X_n \otimes Y_m))] \,.
$$
Thus, we can use $(\id,\id,\eta_\otimes)$ to obtain a natural
transformation $\eta_\otimes^*$ from $\otimes \circ (F_*,F_*)$
to $F_* \circ \otimes$.  This transformation inherits all properties
from $\eta_\otimes$.  In particular, $\eta_\otimes^*$ is lax symmetric
monoidal if $\eta_\otimes$ was so.

It remains to check the properties concerning the distributivity
laws.  As $\dr$ is the identity on the $J$-graded bipermutative
category and on the homotopy colimit, and $\eta_\oplus$ is the
identity on the homotopy colimit, the equalities reduce to
$$
\eta_\otimes^* \oplus \eta_\otimes^* = \eta_\otimes^*
$$
and
$$
F_*(\dl) \circ (\eta_\otimes^* \oplus \eta_\otimes^*)
	= \eta_\otimes^* \circ \dl \,.
$$
The first equation is straightforward to check.

The \leftdistributivity law in the homotopy colimit is given by $\dl =
(\xi,\id,\id)$ and $\eta_\otimes^* \oplus \eta_\otimes^*$ is equal to
$$
\eta_\otimes^* \oplus \eta_\otimes^* = (\id_{\mathbf{nm}},
\id, \eta_\otimes)
\oplus (\id_{\mathbf{nm'}}, \id, \eta_\otimes) \,.
$$
As addition in the homotopy colimit is given by
concatenation, we can simplify the above expression to
$(\id_{\mathbf{nm+nm'}},\id,\eta_\otimes)$.  As $\dl$ differs from
the identity only in the first coordinate, and $\eta_\otimes^* \oplus
\eta_\otimes^*$ only in the third coordinate, these morphisms commute.
\end{proof}

\section{A ring completion device} \label{sec:zeros}

Recall from subsection~\ref{sec:cube} the construction $G\M\:
I\smallint\Q\to\strict$.

\begin{lem} \label{lem:addeqs}
Let $\M$ be a permutative category.  Then
\begin{enumerate}
\item
the canonical functor $\M\to\hocolim_{I\smallint\Q}G\M$ is a stable
equivalence,
\item
$\hocolim_{I\smallint\Q}G\M$ is group complete, and
\item
the canonical functor $\hocolim_{T\in\Q\mathbf1}G\M(\mathbf1,T) \we
\hocolim_{I\smallint\Q}G\M$ is an unstable equivalence.
\end{enumerate}
\end{lem}

\begin{proof}
Recall that spectrification commutes with homotopy colimits \cite{Th3}*{4.1},
\ie, $\hocolim_J \Spt$ is equivalent to $\Spt \hocolim_J$.  Given $\bfn\in
I$, the homotopy colimit $\hocolim_{T\in\Q\bfn} \Spt G\M(\bfn,T)$ can be
calculated by taking the homotopy colimit in each of the $n$ directions of
$\Q\bfn$ successively.  Since all nontrivial maps involved are diagonal
maps, we see that the homotopy colimit in the $n$-th direction can be
identified with $\hocolim_{S\in\Q(\mathbf{n-1})} \Spt
G\M(\mathbf{n-1},S)$, through the inclusion $\mathbf{n-1}\to\bfn$
that skips~$n$.  By induction it follows that each morphism in the
$I$-shaped diagram $\bfn \mapsto \hocolim_{T\in\Q\bfn} G\M(\bfn,T)$ is
a stable equivalence.  Lemma~\ref{lem:boekstedt_lemma} then says that
the functor $\M \to \hocolim_{\bfn\in I} \hocolim_{T\in\Q\bfn} G\M(\bfn,T)$
is a stable equivalence.

The claim that the functor $\M \to \hocolim_{I\smallint\Q} G\M$ is a stable
equivalence follows, since by extending Thomason's proof \cite{Th1} of
$\hocolim_I |\Q| \simeq |I\smallint\Q| = \hocolim_{I\smallint\Q} {*}$
(for the trivial functor~$*$) to allow for arbitrary functors from
$I\smallint\Q$, we have an equivalence
$$
\hocolim_{I\smallint\Q} \Spt G\M \simeq
\hocolim_{\bfn\in I} \hocolim_{T\in\Q\bfn} \Spt G\M(\bfn,T) \,.
$$
See also \cite{Sch}*{2.3} for a write-up in the dual situation.

That $\pi_0$ of $\hocolim_{I\smallint\Q} G\M$ is a group can be seen as
follows.  It is enough to show that elements of the form $1[((\bfn,S),a)]$
have negatives, for $n\geq1$.
If $S \ne \bfn$ then there is an inclusion $S \subseteq T \in \Q\bfn$ with
$T$ containing a negative number, so there is a path
$$
1[((\bfn,S),a)]\to 1[((\bfn,T),0)] \gets 1[((\mathbf0, \mathbf0),0)]
$$
in the homotopy colimit, and the element represents zero.

If $S = \bfn$, so that $a \in \M^{\P\bfn}$, let
$b \in \M^{\P\bfn}$ be given by $b_U = a_V$, where $V = U \cup \{n\}$
if $n \notin U$ and $V = U \setminus \{n\}$ if $n \in U$, for
all $U \subseteq \bfn$.  Then $a \oplus b$ is isomorphic to
$\M_{\bfn}(\iota)(c)$ for some $c \in \M^{\P S}$, where
$S = \mathbf{n-1}$ and $\iota \: S \subset \bfn$ is the inclusion.
Hence there is a path
$$
1[(\bfn,\bfn),a] \oplus 1[(\bfn,\bfn),b]
\to 1[(\bfn,\bfn), a \oplus b]
\leftrightarrow 1[(\bfn,\bfn), \M_{\bfn}(\iota)(c)]
\gets 1[(\bfn,S), c]
$$
in the homotopy colimit, and, as we saw above, the right hand element
represents zero.

Now, since stable equivalences between group complete symmetric
monoidal categories are unstable equivalences, the third claim also follows.
\end{proof}

\begin{lem}\label{lem:compGQ}
If $\M$ is a permutative category with zero, such that all morphisms
are isomorphisms and each additive translation is faithful, then there
is an unstable equivalence
$$
\hocolim_{S\in\Q\mathbf1} G\M(\mathbf1,S) \we (-\M)\M \,.
$$
\end{lem}
\begin{proof}
This is entirely due to Thomason.  Theorem~5.2 in \cite{Th3} asserts
that there is an unstable equivalence from $\hocolim_{S\in\Q\mathbf1}
G\M(\mathbf1,S)$ to the ``simplified double mapping cylinder'', and his
argument on pp.~1657--1658 exhibits an unstable equivalence from the
simplified double mapping cylinder to $(-\M)\M$.
\end{proof}

\begin{rem}
The unstable equivalence $\hocolim_{S\in\Q\mathbf1} G\M(\mathbf1,S)
\we (-\M)\M$ is the additive extension of the assignment that sends
$1[\{-1\},0]$ and $1[\varnothing,a]$ to $(0,0)\in(-\M)\M$, and
$1[\{1\},(a,b)]$ to $(a,b)$.  The map on morphisms is straightforward,
once one declares that the morphism $1[\varnothing,a] \to 1[\{1\},
(a,a)]$ is sent to $[\id_a,a]\: (0,0) \to (a,a) \in (-\M)\M$.
\end{rem}

Collecting Proposition~\ref{prop:GQisIgrded}, Lemma~\ref{lem:zeroless}
and Lemma~\ref{lem:addeqs}, we obtain our zeroless ring completion.

\begin{cor}
Let $\R$ be a \bp (resp.~a \sbc).  The canonical lax morphism
$$
\R \lra \hocolim_{I\smallint\Q}G\R
$$
is a stable equivalence of zeroless \bps (resp.~zeroless \sbcs), and
$$
\hocolim_{S\in\Q\mathbf1}G\R(\mathbf1,S) \we \hocolim_{I\smallint\Q}G\R
$$
is an unstable equivalence of $\R$-modules.
\end{cor}

Using Proposition~\ref{prop:hocolim of rigs} to add zeros, and tracing
the action of $Z\R$, we have the main result:

\begin{thm} \label{thm:gcp}
If $\R$ is a commutative $\bc$ (resp.~a \bc), then
$$
\bar{\R}=\Dhocolim_{I\smallint\Q}G\R
$$
is a simplicial commutative ring category (resp.~a simplicial
ring category).  Here $G\R$ is the $I\smallint\Q$-graded \bp
(resp.~$I\smallint\Q$-graded strictly bimonoidal category) of
Proposition~\ref{prop:GQisIgrded} applied to the bipermutative category
(resp.~strictly bimonoidal category) associated with $\R$.

The simplicial rig maps of Proposition~\ref{prop:hocolim of rigs}
$$
\R \ew Z\R \lra \bar{\R}
$$
are stable equivalences of $Z\R$-modules.  Furthermore, if $\R$ is a
groupoid with faithful additive translation, then the maps
$$
(-\R)\R \ew Z(-\R)\R \ew Z\hocolim_{S\in\Q\mathbf1} G\R(\mathbf1,S)
	\we \bar{\R}
$$
form a chain of unstable equivalences of $Z\R$-modules.
\end{thm}

\begin{comment}

\end{comment}

\end{document}